\documentclass{amsart}
\usepackage{amsmath,amsthm}
\usepackage{amsfonts,amssymb}
\usepackage{accents}
\usepackage{enumerate}
\usepackage{accents,color}
\usepackage{graphicx}

\newcommand{\lvt}{\left|\kern-1.35pt\left|\kern-1.3pt\left|}
\newcommand{\rvt}{\right|\kern-1.3pt\right|\kern-1.35pt\right|}

\newtheorem{thm}{Theorem}[section]
\newtheorem{cor}[thm]{Corollary}
\newtheorem{lem}[thm]{Lemma}
\newtheorem{prop}[thm]{Proposition}

\newtheorem{defn}[thm]{Definition}

\theoremstyle{remark}
\newtheorem{rem}{Remark}[section]

 \def\la{{\langle}}
 \def\ra{{\rangle}}

 \def\d{\mathrm{d}}
 
 \def\sph{{\mathbb{S}^{d-1}}}

 \def\a{{\alpha}}
 \def\b{{\beta}}
 \def\g{{\gamma}}
 \def\k{{\kappa}}
 \def\t{{\theta}}
 \def\l{{\lambda}}
 \def\o{{\omega}}
 \def\s{\sigma}
 \def\la{{\langle}}
 \def\ra{{\rangle}}

 \def\kb{{\mathbf k}}

 \def\Kb{{\mathbf K}}
 \def\Pb{{\mathbf P}}
 
 \def\CD{{\mathcal D}}
 
 \def\CH{{\mathcal H}}

 \def\CP{{\mathcal P}}

 \def\CV{{\mathcal V}}

 \def\BB{{\mathbb B}}

 \def\NN{{\mathbb N}}

 \def\RR{{\mathbb R}}
 \def\SS{{\mathbb S}}
 \def\VV{{\mathbb V}}
 \def\TT{{\mathbb T}}
 \def\ZZ{{\mathbb Z}}
      \def\proj{\operatorname{proj}}

\def\lla{\langle{\kern-2.5pt}\langle}      
\def\rra{\rangle{\kern-2.5pt}\rangle}

\newcommand{\wh}{\widehat}

\def\f{\frac}

\graphicspath{{./}}
\begin{document}

\title{Orthogonal polynomials and Fourier orthogonal series on a cone}

\author{Yuan Xu}
\address{Department of Mathematics\\ University of Oregon\\ 
Eugene, Oregon 97403-1222.}\email{yuan@uoregon.edu}

\thanks{The author was supported in part by NSF Grant DMS-1510296.}

\date{\today}

\begin{abstract} 
Orthogonal polynomials and the Fourier orthogonal series on a cone in $\RR^{d+1}$ are studied. It is shown
that orthogonal polynomials with respect to the weight function $(1-t)^\g (t^2-\|x\|^2)^{\mu-\f12}$
on the cone $\VV^{d+1} = \{(x,t): \|x\| \le t \le 1\}$ are eigenfunctions of a second order differential operator, with 
eigenvalues depending only on the degree of the polynomials, and the reproducing kernels of these polynomials 
satisfy a closed formula that has a one-dimensional chacteristic. The latter leads to a convolution structure on the 
cone, which is then utilized to study the Fourier orthogonal series. This narrative also holds, in part, for more 
general classes of weight functions. Furthermore, analogous results are also established for orthogonal structure
on the surface of the cone. 
\end{abstract}

\maketitle

\section{Introduction}
\setcounter{equation}{0}

Given an appropriate weight function $\varpi$ defined on a domain $\Omega$ in $\RR^d$, one can construct 
an orthogonal basis of polynomials in $L^2(\Omega, \varpi)$ and study the Fourier series in orthogonal 
polynomials. For study beyond the abstract Hilbert space setting, it is necessary to consider special weight 
functions on regular domains to ensure manageable structure for orthogonal polynomials. 

The situation is best illustrated by the Fourier series in spherical harmonics, which are orthogonal with respect
to the Lebesgue measure on the unit sphere $\sph$ of $\RR^d$. The key ingredient for the 
Fourier analysis on the sphere is the addition formula of spherical harmonics (cf. \cite{DaiX,Stein}), which 
shows that the reproducing kernel of spherical harmonics of degree $n$ is equal to $Z_n(\la x,y\ra)$, where 
$Z_n$ is an ultraspherical polynomial of degree $n$. This closed-form expression of one-dimensional structure is 
utilized for just about every aspect of analysis on the unit sphere. Other most studied cases of orthogonal 
polynomials in $d$-variables are the classical ones in the following list \cite{DX}:
\begin{enumerate}[\, (1)]
\item Hermite:   $\varpi(x) = e^{-\|x\|^2}$ on $\RR^d$;
\item Laguerre:  $\varpi_\a(x) = x^\a e^{-|x|}$ on $\RR_+^d$; 
\item Ball:    $\varpi_\mu(x)= (1-\|x\|^2)^{\mu-\f12}$ on $\BB^d =\{x: \|x\| \le 1\}$;
\item Simplex: $\varpi_\k(x) = x_1^{\k_1} \ldots x_d^{\k_d} (1-|x|)^{\k_{d+1}}$ 
    on $\TT^d = \{x: x_1, \ldots, x_d \ge 0, |x| \le 1\}$, 
\end{enumerate}
where $|x| = x_1+\cdots + x_d$ and the parameters $\a\in \RR^d$, $\mu \in \RR$ and $\k \in \RR^{d+1}$ 
are chosen so that 
the weight functions are integrable. In each of these cases, several orthogonal polynomial bases can be 
constructed explicitly and they are classical, some of them could be traced back to the work of Hermite 
\cite{AF}. The Fourier series of classical orthogonal polynomials have been studied in depth in recent years; 
see \cite{DaiX, DX, Thanga} and references therein. In the case of the ball and the simplex, the essential 
ingredient of the study is a closed-form expression for the reproducing kernels of orthogonal polynomials, 
which can be regarded as analogues of the addition formula of spherical harmonics. 
  
The purpose of this paper is to study orthogonal polynomials and Fourier orthogonal series on a cone of 
revolution, which we assume to be 
$$
   \VV^{d+1} := \left \{ (x,t) \in \RR^{d+1}:  \|x\| \le t, \,\,  0 \le t \le b, \,\, x \in \RR^d \right\}, 
$$ 
and also on the surface of the cone 
$$
\VV_0^{d+1}:= \left \{ (x,t) \in \RR^{d+1}:  \|x\| = t, \,\,  0 \le t \le b, \,\, x \in \RR^d \right\} 
$$
for $d \ge 2$. For $d =1$, the cone $\VV^2$ is a 
triangle on the plane. For $d \ge 2$, orthogonal structure on these domains have not been studied in the 
literature as far as we are aware. On the surface of the cone $\VV_0^{d+1}$, 
we will define orthogonality with respect to $\varphi(t)\d \s(x,t)$, where $\s(x,t)$ is the surface measure of 
$\VV_0^{d+1}$. On the solid cone $\VV^{d+1}$ we will define the orthogonality with respect to an 
appropriate weight function $W$, which we choose to be of the form
$$
   W(x,t) = w(t) (t^2 - \|x\|^2)^{\mu-\f12}, \qquad \mu > -\tfrac12, \quad x\in \BB^d, \quad 0 \le t \le b,
$$ 
making use of the classical weight function $\varpi_\mu$ on the unit ball. For each $\varphi$ or $w$, we 
can construct orthogonal polynomials with respect to $W$ explicitly. We shall, however, choose special
$\varphi$ and $w$ so that the corresponding orthogonal polynomials possess two characteristic properties
of classical orthogonal polynomials. 

The first property is that there is a second order differential operator $\CD$, which is linear and has polynomial
coefficients, such that orthogonal polynomials are the eigenfunctions of $\CD$ with eigenvalues depending only 
on the degree of the polynomials. For non-degenerate inner product on a solid domain in two-variables, a classification 
of orthogonal polynomials with such a property in \cite{KS} shows that, up to affine transformation and for 
positive definite integrals, there are only five families, and they are the four classical ones, (1)--(4), together 
with the product of Hermite-Laguerre on $\RR \times \RR_+$. For more than two-variables, no classification is 
known but the four classical families and some of their tensor products have this property \cite{DX}. As we 
shall show in Section 3, two families of orthogonal polynomials on the cone $\VV^{d+1}$ for $d \ge 2$ 
possess this property, which are orthogonal with respect to $W(x,t)$ with $w(t) = (1-t)^\g$ and with 
$w(t) = e^{-t}$. Furthermore, for spherical harmonics, the property is satisfied with $\CD$ being the 
Laplace-Beltrami operator on $\sph$. Our results give two families of orthogonal polynomials on the surface 
of the cone that possess this property, which are orthogonal with respect to $\varphi(t) \d \s(x,t)$ with, 
somewhat surprisingly, $\varphi(t) = t^{-1} (1-t)^\g$ and $\varphi(t) = t^{-1} e^{-t}$, both of which have
a singularity at the apex of the cone. 
The second property is a closed-form expression for the reproducing kernel of orthogonal polynomials, akin to 
the addition formula of spherical harmonics. On the unit ball and on the simplex, the Fourier series in classical 
orthogonal polynomials have been studied actively in recent years (cf. \cite{DaiX,DX, KPX, PX} and the 
references therein), after such a formula was discovered about twenty years ago \cite{X97,X99}. This 
property holds for orthogonal polynomials on compact domains. We consider the cone $\VV^{d+1} =
 \{(x,t): \|x\| \le t, \, x \in \RR^d, \, t \le 1\}$, which is compact, and establish closed-form formulas for the 
reproducing kernel for orthogonal polynomials on $\VV^{d+1}$ with $w(t) = t^\b (1-t)^\g$, $\b \ge 0$ and 
$\g \ge -\f12$ and for $\varphi(t) = t^\b(1-t)^\g$ for $\b \ge -1$ and $\g \ge -\f12$ on $\VV_0^{d+1}$. 
In both cases, the closed-form formulas provide a one-dimensional structure 
for the kernel, which allows us to define a convolution structure on the cone. The projection operator of 
the Fourier orthogonal series of a function $f$ on the cone can be written as the convolution of $f$ and
the Jacobi kernel of one-variable. As a consequence, many properties of the Fourier orthogonal series on
the cone can be deduced from the Fourier-Jacobi series. As an example, we shall deduce sharp 
conditions for the convergence of the Ces\`aro means of the Fourier orthogonal series on the cone.  

The paper is organized as follows. In the next section, we recall definitions of orthogonal polynomials
of several variables and the Fourier orthogonal series, as well as several specific families of orthogonal
polynomials that will be used later. The orthogonal structure on the solid cone will be studied in 
Section 3 and the closed formula for the reproducing kernel will be established in Section 4, which leads 
to a translation operator and a convolution structure that will be discussed in Section 5. These are then
used in Section 6 to study the summability of the Fourier orthogonal series on the cone. Our development 
of orthogonal structure on the surface of the cone is parallel to that on the solid cone. Orthogonal 
structure on the surface will be studied in Section 7 and the closed formulas for the reproducing kernels 
will be derived in Section 8. The convolution structure on the surface and its application in the Fourier 
orthogonal series on the surface will be discussed in Section 9. Finally, our results could be extended to 
more general weight functions that include Dunkl's reflection invariant weight function as a factor, which
will be briefly discussed in Section 10. 

\section{Preliminary}
\setcounter{equation}{0}
We provide general framework on orthogonal polynomials and the Fourier orthogonal series in the first 
subsection and review several families of specific orthogonal polynomials that will be needed in 
several subsequent subsections. 

\subsection{Orthogonal polynomials and Fourier orthogonal expansion}
A general reference for this subsection is \cite[Chapter 3]{DX}. Let $W$ be a nonnegative weight 
function defined on a domain $\Omega$ of positive measure
in $\RR^d$ so that $\int_\Omega W(x) \d x > 0$, and let $\la \cdot,\cdot\ra$ be an inner product 
defined by 
$$
       \la f,g\ra = b_W \int_\Omega f(x) g(x) W(x) dx,
$$
where $b_W$ is a normalization constant such that $\la 1,1\ra =1$. Denote by $\Pi^d$ the space of
polynomials of $d$-variables and by $\Pi_n^d$ the subspace of polynomials of degree at most $n$
in $\Pi^d$. A polynomial $P \in \Pi_n^d$ is called orthogonal with respect to the inner product 
$\la\cdot,\cdot \ra$ if $\la P, Q\ra =0$ for all polynomials $Q\in \Pi_{n-1}^d$. Let $\CV_n^d$ be the 
space of orthogonal polynomials of degree $n$. It is known that 
$$
   \dim \Pi_n^d = \binom{n+d}{n} \quad\hbox{and}\quad \dim \CV_n^d = \binom{n+d-1}{n}.
$$

Let $\NN$ be the set of natural integers and let $\NN_0= \NN \cap \{0\}$. The space $\CV_n^d$ 
can have many distinct bases. Let $\{P_{\kb}^n :|\kb| = n, \, \kb \in \NN_0^d\}$
be an orthogonal basis of $\CV_n^d$. For $f \in L^2(\Omega)$, the Fourier orthogonal series is 
defined by 
$$
   f = \sum_{n=0}^\infty \sum_{|\kb| = n} \wh f_{\kb}^n P_{\kb}^n = \sum_{n=0}^\infty \proj_n f,
   \qquad \hbox{where} 
      \quad \wh f_{\kb}^n = \frac{\la f, P_{\kb}^n\ra} {\la P_{\kb}^n, P_{\kb}^n\ra},
$$
and $\proj_n: L^2 \mapsto \CV_n^d$ denotes the orthogonal projection operator, which can be written
as an integral against the reproducing kernel $\Pb_n$ of the space $\CV_n^d$, 
\begin{equation} \label{eq:projOP}
   \proj_n f(x) = \sum_{|\kb| = n}  \wh f_{\kb}^n P_{\kb}^n= \int_\Omega f(y) \Pb_n(x,y) W(y)\d y. 
\end{equation}
In terms of the basis $\{P_{\kb,n}:|\kb| = n\}$, the reproducing kernel can be written as 
\begin{equation} \label{eq:reprodOP}
  \Pb_n(x,y) = \sum_{|\kb| = n} \frac{P_{\kb}^n(x) P_{\kb}^n(y)} {\la P_{\kb}^n, P_{\kb}^n\ra}.
\end{equation}
The kernel, however, is uniquely defined and is independent of the choice of the orthogonal basis. 

For our study of orthogonal structure in and on the cone, we shall need various properties of spherical
harmonics, orthogonal polynomials on the unit ball in $\RR^d$ and on the triangle $\TT^2$, as well
as the Jacobi polynomials and Gegenbauer polynomials on $[-1,1]$. We shall collect what will be 
needed in subsequent subsections. 

\subsection{Classical orthogonal polynomials}
We fix notations that will be used throughout this paper. For $\a,\b > -1$, the Jacobi weight function 
is denoted by 
$$
   w_{\a,\b} (x) = (1-x)^\a (1+x)^\b, \qquad - 1< x< 1.
$$
Its normalization constant $c'_{\a,\b}$, defined by $c'_{\a,\b}  \int_{-1}^1 w_{\a,\b} (x)dx = 1$, is given by
\begin{equation}\label{eq:c_ab}
   c'_{\a,\b} = \frac{1}{2^{\a+\b+1}} c_{\a,\b} \quad\hbox{with} \quad c_{\a,\b} := \frac{\Gamma(\a+\b+2)}{\Gamma(\a+1)\Gamma(\b+1)}.
\end{equation}
The Jacobi polynomials $P_n^{(\a,\b)}$ are hypergeometric functions given by \cite[(4.21.2)]{Sz}
$$
  P_n^{(\a,\b)}(x) = \frac{(\a+1)_n}{n!} {}_2F_1 \left (\begin{matrix} -n, n+\a+\b+1 \\\
      \a+1 \end{matrix}; \frac{1-x}{2} \right),
$$
and they are orthogonal polynomials with respect to $w_{\a,\b}$ satisfying 
$$
  c_{\a,\b}' \int_{-1}^1 P_n^{(\a,\b)}(x) P_m^{(\a,\b)}(x) w_{\a,\b}(x) \d x = h_n^{(\a,\b)} \delta_{n,m},
$$
where 
\begin{align} \label{eq:JacobiNorm}
  h_n^{(\a,\b)} = \frac{(\a+1)_n (\b+1)_n(\a+\b+n+1)}{n!(\a+\b+2)_n(\a+\b+2 n+1)}.
\end{align}
The polynomial $P_n^{(\a,\b)}$ satisfies a differential equation \cite[(4.2.1)]{Sz}
\begin{equation} \label{eq:JacobiDE}
 (1-x^2) u'' - (\a-\b+(\a+\b+2)x )u' + n (n+\a+\b+1)u =0.
\end{equation}

We will use the Gegenbauer polynomials $C_n^\l$, also called ultraspherical polynomials when $\l$ is
half-integer. For $\l > -\f12$, let $w_\l$ be the weight function
$$
  w_\l(x) = (1-x^2)^{\l-\f12}, \quad -1< x<1.
$$
Its normalization constant $c_{\l}$, defined by $c_{\l}  \int_{-1}^1 w_{\l} (x)dx = 1$, is given by
\begin{equation}\label{eq:c_l}
    c_{\l} := \frac{\Gamma(\l+1)}{\Gamma(\f12)\Gamma(\l+\f12)}.
\end{equation}
The Genenbauer polynomials $C_n^\l$ are orthogonal with respect to $w_\l$,
\begin{equation} \label{eq:GegenNorm}
  c_{\l} \int_{-1}^1 C_n^{\l}(x) C_m^{\l}(x) w_\l(x) \d x = h_n^{\l} \delta_{n,m}, \qquad h_n^\l = \frac{\l}{n+\l}C_n^\l(1),
\end{equation}
and are normalized so that $C_n^\l(1) = \frac{(2\l)_n}{n!}$. The polynomial $C_n^\l$ is a constant multiple of 
the Jacobi polynomial $P_n^{(\l-\f12,\l-\f12)}$ and it is also related to the Jacobi polynomial by a quadratic transform \cite[(4.7.1)]{Sz}
\begin{equation} \label{eq:Jacobi-Gegen}
  C_{2n}^\l (x) = \frac{(\l)_n}{(\f12)_n} P_n^{(\l-\f12,\l-\f12)}(2x^2-1).
\end{equation}

For convenience, see \eqref{eq:additionF} below, we shall define 
\begin{equation} \label{eq:Zn} 
Z_n^\l(t):=  \frac{C_n^\l(1)C_n^\l(t)}{h_n^\l} = \frac{n+\l}{\l} C_n^\l(t), \quad \l >  0, \qquad Z_n^0(t):= \begin{cases} 2 T_n(t), & n \ge 1 \\ 1, & n =0 \end{cases},
\end{equation}
where $T_n(x)  = \arccos (\cos x)$ is the Chebyshev polynomial, and use this notation throughout the rest of this paper.

We will also need the Laguerre polynomials $L_n^\a$. For $\a > -1$, these polynomials are orthogonal
with respect to $x^\a e^{-x}$ on $\RR_+ = [0,\infty)$,
\begin{equation} \label{eq:LaguerreNorm}
  \frac{1}{\Gamma(\a+1)} \int_{0}^\infty L_n^{\a}(x) L_m^{\a}(x) x^\a e^{-x} \d x =  \frac{(\a+1)_n}{n!} \delta_{n,m}.\end{equation}

\subsection{Spherical harmonics}
There are maby sources for this topic, we follow \cite[Chapter 1]{DaiX}. Let $\CP_n^d$ denote the 
space of homogeneous polynomials of degree $n$ in $d$-variables. 
A spherical harmonics $Y$ of degree $n$ is an element of $\CP_n^d$ that satisfies $\Delta Y =0$, 
where $\Delta$ is the Laplace operator of $\RR^d$. If $Y \in \CP_n^d$, then $Y(x) = \|x\|^n Y(x')$, 
$x' = x/\|x\| \in \sph$, so that $Y$ is determined by its restriction on $\sph$. Let $\CH_n^d$ denote
the space of spherical harmonics of degree $n$. Then 
\begin{equation} \label{eq:sphHn-dim}
 \dim \CH_n^d =\binom{n+d-1}{n} - \binom{n+d-3}{n-2} = \binom{n+d-2}{n} +\binom{n+d-3}{n-1}, 
\end{equation}   
and we also have $\dim \CP_n^d = \binom{n+d-1}{n}$.   
Spherical harmonics of different degrees are orthogonal on the sphere. For $n \in \NN_0$ let
$\{Y_\ell^n: 1 \le \ell \le \dim \CH_n^d\}$ be an orthonormal basis of $\CH_n^d$ in this subsection; then 
$$
   \frac{1}{\o_d} \int_\sph Y_\ell^n (\xi) Y_{\ell'}^m (\xi) d\s(\xi) = \delta_{\ell,\ell'} \delta_{m,n},
$$
where $\o_d$ denotes the surface area $\o_d = 2 \pi^{\f{d}{2}}/\Gamma(\f{d}{2})$ of $\sph$. Let $\Delta_0$
be the Laplace--Beltrami operator on the sphere, which is the restriction of $\Delta$ on the 
unit sphere. 
Spherical harmonics are eigenfunctions of this operator \cite[(1.4.9)]{DaiX}, 
\begin{equation} \label{eq:sph-harmonics}
     \Delta_0 Y = -n(n+d-2) Y, \qquad Y \in \CH_n^d;
\end{equation}
the eigenvalues depend only on the degree $n$ of $\CH_n^d$. Moreover, spherical harmonics 
satisfy an addition formula \cite[(1.2.3) and (1.2.7)]{DaiX},
\begin{equation} \label{eq:additionF}
   \sum_{\ell =1}^{\dim \CH_n^d} Y_\ell^n (x) Y_\ell^n(y) = Z_n^{\f{d-2}{2}} (\la x,y\ra), \quad x, y \in \sph, 
\end{equation}
where $Z_n^\l$ is defined in \eqref{eq:Zn}. For $f \in L^2(\sph)$, its Fourier orthogonal series in spherical
harmonics is defined by 
$$
f = \sum_{n=0}^\infty \proj_n f = \sum_{n=0}^\infty \sum_{\ell=0}^{ \dim \CH_n^d} \wh f_\ell^n Y_\ell^n (\xi),    
   \qquad  \wh f_\ell^n = \frac{1}{\o_d} \int_\sph f(y) Y_\ell^n(y) \d \s, 
$$
where the orthogonal projection $\proj_n: L^2(\sph) \to \CH_n^d$ can be written as an integral
$$
   \proj_n f(x) = \sum_{\ell = 1}^{\dim \CH_n^d} \wh f_{\ell,n} Y_\ell^n  
     = \frac{1}{\o_d}\int_{\sph} f(y) \Pb_n(x,y) \d\s(y')
$$
in terms of the reproducing kernel $\Pb_n(x,y) = \sum_\ell Y_\ell^n(x)Y_\ell^n(y)$. The addition formula
\eqref{eq:additionF} provides a closed formula for the reproducing kernel and shows that it has a one-dimensional
structure. This formula plays an essential role is the study of Fourier orthogonal series on the sphere; see,
for example, \cite{DaiX, Stein}. 

\subsection{OP on the unit ball} 
Classical orthogonal polynomials on the unit ball $\BB^d$ are orthogonal with respect to the inner product
$$
  \la f,g\ra_{\mu} =b_\mu^B \int_{\BB^d} f(x) g(x) \varpi_\mu(x) dx,
$$
where $\varpi_\mu$ is the weight function 
\begin{equation}\label{eq:weightB}
 \varpi_\mu(x):= (1-\|x\|^2)^{\mu-\f12},  \quad \mu > -\tfrac12, \quad \hbox{and} \quad 
     b_\mu^B = \frac{\Gamma(\mu+\f{d+1}{2})}{\pi^{\f{d}{2}}\Gamma(\mu+\f12)}
\end{equation}
is the normalization 
constant so that $\la 1,1\ra_\mu =1$. 
Let $\CV_n^d(\varpi_\mu)$ be the space of orthogonal polynomials of degree $n$ with respect to $\varpi_\mu$.
An orthogonal basis of $\CV_n^d(\varpi_\mu)$ can be given explicitly in terms of the Jacobi polynomials and
spherical harmonics. For $ 0 \le m \le n/2$, let $\{Y_\ell^{n-2m}: 1 \le \ell \le \dim \CH_{n-2m}^d\}$ be an 
orthonormal basis of $\CH_{n-2m}^d$. Define \cite[(5.2.4)]{DX}
\begin{equation}\label{eq:basisBd}
  P_{\ell, m}^n (x) = P_m^{(\mu-\f12, n-2m+\f{d-2}{2})} \left(2\|x\|^2-1\right) Y_{\ell, n-2m}(x).
\end{equation}
Then $\{P_{\ell,m}^n: 0 \le m \le n/2, 1 \le \ell \le \dim \CH_{n-2m}^d\}$ is an orthogonal basis of 
$\CV_n^d(\varpi_\mu)$. For other explicit orthogonal bases of $\CV_n^d(\varpi)$, see \cite[Chapter 5]{DX}. 

As in the case of spherical harmonics, classical orthogonal polynomials on the unit ball are eigenfunctions 
of a second order differential operator: for $u \in \CV_n^d(\varpi_\mu)$,  
\begin{equation}\label{eq:diffBall}
  \left( \Delta  - \la x,\nabla \ra^2  - (2\mu+d-1) \la x ,\nabla \ra \right)u = - n(n+2\mu+ d-1) u.
\end{equation}
Let $\Pb_n(\varpi_\mu;\cdot,\cdot)$ be the reproducing kernel of the space $\CV_n^d(\varpi_\mu)$, as defined
in \eqref{eq:reprodOP}. Then the kernel satisfies a closed formula \cite[(5.2.7)]{DX} for $\mu \ge 0$,
\begin{align}\label{eq:PnBall}
 \Pb_n(\varpi_\mu;x,y) = c_{\mu-\f12}
      \int_{-1}^1 Z_n^{\mu+\f{d-1}{2}} & \left(\la x,y\ra+ t \sqrt{1-\|x\|^2}\sqrt{1-\|y\|^2} \right) \\
       &  \times  (1-t^2)^{\mu-1} \d t, \notag
\end{align}
where $c_{\mu-\f12}$ is give by \eqref{eq:c_l} and the identity holds when $\mu =0$ under the limit 
\begin{equation}\label{eq:limit-int}
  \lim_{\mu \to 0}  c_{\mu-\f12} \int_{-1}^1 f(t) (1-t^2)^{\mu-1} dt = \frac{f(1) + f(-1)}{2}.  
\end{equation}

\subsection{Jacobi polynomials on the triangle} 
In two variables, the cone $\VV^2$ reduces to a triangle, but not the standard triangle 
$\TT^2 = \{(x_1,x_2): x_1, x_2 \ge 0, x_1+x_2 \le 1\}$. The classical orthogonal polynomials on the 
triangle are orthogonal with respect to the weight function
$$
  \varpi_{\a,\b,\g}(x_1,x_2) = x_1^{\a} x_2^{\b} (1-x_1-x_2)^\g, \qquad \a,\b,\g > -1, 
$$
on $\TT^2$. Let $\CV_n^2( \varpi_{\a,\b,\g})$ denote the space of orthogonal polynomials with respect 
to $ \varpi_{\a,\b,\g}$. Several bases of $\CV_n^2( \varpi_{\a,\b,\g})$ can be given explicitly in terms of 
the Jacobi polynomials. One of them consists of \cite[Prop. 2.4.2]{DX}
\begin{equation} \label{eq:triangleOP}
  P^{\a,\b,\g}_{k,n} (x_1,x_2) = P_{n-k}^{(2k+\a+\b+1,\g)}(1-2x_1-2x_2) (x_1+x_2)^k P_k^{(\b,\a)} \left(\frac{x_1-x_2}{x_1+x_2}\right) 
\end{equation}
for $0 \le k \le n$. The $L^2$ norms of these polynomials are given by
\begin{align} \label{eq:triangleOPnorm}
    h_{k,n}^{(\a,\b,\g)} = \frac{c_{\a+\b+1,\g}}{c_{\a+\b+2k+1,\g}} h_{n-k}^{(\a+\b+2k+1,\g)} h_k^{(\b,\a)}, 
\end{align}
where $c_{\a,\b}$ is given by \eqref{eq:c_ab} and $h_n^{(\a,b)}$ is given by \eqref{eq:JacobiNorm}. 

In terms of this basis, the reproducing kernel $\Pb_n( \varpi_{\a,\b,\g}; \cdot,\cdot)$ is given by 
\begin{align} \label{eq:trianglePn} 
 \Pb_n( \varpi_{\a,\b,\g}; x,y) = & \sum_{k=0}^n 
    \frac{P_{n-k}^{(2k+\a+\b+1,\g)}(1-2x_1-2x_2)P_{n-k}^{(2k+\a+\b+1,\g)}(1-2y_1-2y_2)}{h_{k,n}^{(\a,\b,\g)}} \notag \\
       & \times (x_1+x_2)^k (y_1+y_2)^k P_k^{(\b,\a)} \left(\frac{x_1-x_2}{x_1+x_2}\right)
         P_k^{(\b,\a)} \left(\frac{y_1-y_2}{y_1+y_2}\right). 
\end{align}
For $\a,\b,\g \ge - \frac12$, the kernel satisfies a closed formula in a triple integral \cite[(5.3.5)]{DX}
\begin{align} \label{eq:trianglePn2} 
  \Pb_n( \varpi_{\a,\b,\g}; x,y) = c_{\a}c_{\b}c_{\g}& \int_{[-1,1]^3}  Z_{2n}^{\a+\b+\g+2} (\eta(x,y,t)) \\
      &  \times (1-t_1^2)^{\a-\f12} (1-t_2^2)^{\b-\f12} (1-t_3^2)^{\g-\f12}  \d t, \notag
\end{align}
where 
\begin{align} \label{eq:zetaT} 
  \eta(x,y,t) = \sqrt{x_1 y_1} \,t_1 +  \sqrt{x_2 y_2}\, t_2 + \sqrt{(1-x_1-x_2) (1-y_1-y_2)} \,t_3. 
\end{align}
 
For further results and discussions on classical orthogonal polynomials on the ball and the triangle, see
Sections 5.2, 2.4 and 5.3 of \cite{DX}. 

\section{Orthogonal polynomials on the cone}
\setcounter{equation}{0}

We consider orthogonal polynomials on the cone $\VV^{d+1}$ with respect to the weight function
$$
  W(x,t) = w(t) (t^2-\|x\|^2)^{\mu-\f12}, \qquad \|x\| \le t, \,\, 0 \le t \le b, 
$$
for two families of $w$. We call the first case Jacobi polynomials on the cone, for which $w(t) = (1-t)^\b$,
and call the second case Laguerre polynomials on the cone, for which $w(t) = e^{-t}$. Each of these two 
families are eigenfunctions of a second order differential operator with eigenvalues depending only on 
the total degree of the polynomials.

\subsection{Jacobi polynomials on the cone}
We consider the weight functions 
\begin{align} \label{eq:W-cone}
   W_{\mu,\b,\g}(x,t): = (t^2-\|x\|^2)^{\mu-\f12} t^\b (1-t)^\g, \quad   \mu > -\tfrac12, \, \b> -1,\, \g > -1, 
\end{align}
defined on the solid cone 
$$
\VV^{d+1} =\{(x,t): \|x\| \le t \le 1, \, x \in \BB^d\}
$$ 
and define the inner product 
$$
   \la f,g \ra_{\mu,\b, \g} : = b_{\mu,\b,\g} \int_{\VV^{d+1}} f(x,t) g(x,t) W_{\mu,\b,\g}(x,t) \d x\, \d t,
$$
where $b_{\mu,\b,\g}$ is the constant chosen so that $\la 1, 1 \ra_{\mu,\b,\g} =1$. For $(x,t) 
\in \VV^{d+1}$, we let $x = t u$ with $u \in \BB^d$, so that $\d x = t^d \d u$ and
\begin{align} \label{eq:int-V}
  \int_{\VV^{d+1}} f(x,t) W_{\mu,\b,\g} (x,t) \d x \, \d t \, & = \int_0^1 \int_{\|x\| \le t} f(x,t) (t^2- \|x\|^2)^{\mu-\f12} t^\b (1-t)^\g 
     \d x \,\d t  \notag \\
      & = \int_0^1 \int_{\BB^d} f(t u, t) t^{2\mu+\b+d-1}  (1-t)^\g\d t  (1-\|u\|^2)^{\mu -\f12} \d u, 
\end{align}
which implies, in particular, that
$$
b_{\mu,\b,\g} = c_{2\a,\g} b_\mu^B, \qquad \a = \mu+\frac{\b+d-1}{2},
$$
where $c_{\a,\b}$ is defined in \eqref{eq:c_ab} and $b_\mu^B$ is the normalized constant of 
$\varpi_\mu$ on $\BB^d$. Let $\CV_n^{d+1}(W_{\mu,\b,\g}) := \CV_n(\VV^{d+1}, W_{\mu,\b, \g})$ 
be the space of orthogonal polynomials of degree exactly $n$ with respect to the inner product 
$\la \cdot,\cdot\ra_{\mu,\b,\g}$.  

\begin{prop}\label{prop:OPconeJ} 
For $m =0,1,2,\ldots$, let $\{P_\kb^m(\varpi_\mu): |\kb| = m, \kb \in \NN_0^d\}$ denote an orthonormal basis 
of $\CV_m^d(\varpi_\mu)$ on the unit ball $\BB^d$. Define $\alpha: = \mu+\frac{\b+d-1}{2}$ and 
\begin{equation} \label{eq:coneJ}
  Q_{m,\kb}^n(x,t):= P_{n-m}^{(2\a+2m, \g)}(1- 2t) t^m P_\kb^m\left(\varpi_\mu; \frac{x}{t}\right).
\end{equation}
Then $\{Q_{m,\kb}^n(x,t): |\kb| = m, \, 0 \le m\le n\}$ is an orthogonal basis of $\CV_n(\VV^{d+1},W_{\mu,\b,\g})$ 
and the norm of $Q_{m, \kb}^n$ is given by 
\begin{equation} \label{eq:coneJnorm}
    H_{m,n}^{(\a,\g)}:=  \la Q_{m, \kb}^n, Q_{m, \kb}^{n} \ra_{\mu,\b,\g} 
        =  \frac{ c_{2\a,\g}} {c_{2\a+2m,\g} } h_{n-m}^{(2\a+2m, \g)},
\end{equation}
where $h_m^{(\a,\g)}$ is the norm square of the Jacobi polynomial in \eqref{eq:JacobiNorm} 
and $c_{\a,\b}$ is in \eqref{eq:c_ab}.
\end{prop}

\begin{proof}
It is clear that $Q_{m, \kb}^n$ is a polynomial of degree $n$ in $(x,t)$. Using \eqref{eq:int-V} to separate 
variables, it is easy to see that $ \la Q_{m, \kb}^n, Q_{m', \kb'}^{n'} \ra_{\mu,\b,\g} = H_{m,n} 
\delta_{n,n'} \delta_{m,m'} \delta_{\ell,\ell'}$. Furthermore, since $P_{\kb}^m$ is 
orthonormal, we obtain
\begin{align*}
 \la Q_{m, \kb}^n, Q_{m, \kb}^{n} \ra_{\mu,\b,\g} \, & = 
  c_{\a,\g} \int_0^1 \left|P_{n-m}^{(2\a+2m, \g)}(1-2t)\right|^2 t^{2\a+2m} 
           (1-t)^\g\d t  \\
           & = \frac{ c_{2\a,\g}} {c_{2\a+2m,\g} } h_{n-m}^{(2\a+2m, \g)}, 
\end{align*}
since $h_n^{(\a,\g)}$ is defined with respect to the normalized weight and $\frac{c_{\a,\b}}{2^{\a+\b+1}}$ is
the normalization constant of the Jacobi weight on $[-1,1$]. Finally, the cardinality of the set 
$\{Q_{m,\kb}^n(x,t): |\kb| = m, \, 0 \le m\le n\}$ is equal to 
$$
 \sum_{m=0}^n \dim \CV_m^d(\varpi_\mu) = \sum_{m=0}^n \binom{m+d-1}{m} = \binom{n+d}{d} = 
 \dim \CV_n(\VV^{d+1},W_{\mu,\b,\g}),
$$
which verifies that the set is an orthonormal basis of $\CV_n(\VV^{d+1},W_{\mu,\b,\g})$.
\end{proof}

It is not difficult to see that the structure of the orthogonal basis in \eqref{eq:coneJ} holds for 
a generic weight function $w$ instead of the Jacobi weight; see \cite{OX3}. However, it is the 
polynomials in \eqref{eq:coneJ} that share properties of classical orthogonal polynomials on the unit ball. 

The weight function $W_{\mu,\b,\g}$ is a combination of the Jacobi weight on $[0,1]$ and the classical 
weight function on the unit ball. From the point of view of differential equations satisfied by the classical 
orthogonal polynomials, the case $\b = 0$ stands out. We shall define $W_{\mu,\g} = W_{\mu,0,\g}$. 
More explicitly, 
\begin{align} \label{eq:W-mgCone}
   W_{\mu,\g}(x,t) = (t^2-\|x\|^2)^{\mu-\f12} (1-t)^\g, \quad  \mu > -\tfrac12, \, \g > -1. 
\end{align}
Our next theorem shows that $\CV_n(\VV^{d+1},W_{\mu,\g})$ is an eigenspace of a second order 
differential operator. 

\begin{thm} \label{thm:DEconeJ} 
Let $\mu > -\tfrac12$, $\g > -1$ and $n \in \NN_0$. Then every $u \in \CV_n(\VV^{d+1},W_{\mu,\g})$ 
satisfies the differential equation
\begin{equation}\label{eq:cone-eigen}
   \CD_{\mu,\g} u =  -n (n+2\mu+\g+d) u,
\end{equation}
where $\CD_{\mu,\g} = \CD_{\mu,\g}(x,t)$ is the second order linear differential operator
\begin{align*}
  \CD_{\mu,\g} : = & \, t(1-t)\partial_t^2 + 2 (1-t) \la x,\nabla_x \ra \partial_t + \sum_{i=1}^d(t - x_i^2) \partial_{x_i}^2
        - 2 \sum_{i<j } x_i x_j \partial_{x_i} \partial_{x_j}  \\ 
  &   + (2\mu+d)\partial_t  - (2\mu+\g+d+1)( \la x,\nabla_x\ra + t \partial_t),  
\end{align*}
where $\nabla_x$ and $\Delta_x$ denote the gradient and the Laplace operator in $x$-variable.
\end{thm} 
 
\begin{proof}
By Proposition \ref{prop:OPconeJ}, it suffices to establish the result for $u = Q_{m,\kb}^n$ in \eqref{eq:coneJ}. 
For simplicity, we write 
$$
u(x,t) = g(t) H(x,t), \quad  g(t) = P_{n-m}^{(2\mu+2m+d-1, \g)}(1- 2t),
     \quad H(x,t) = t^m P_\kb^m\left(\frac{x}{t}\right). 
$$ 
The differential equation for the Jacobi polynomial \eqref{eq:JacobiDE} shows that $g$ satisfies 
\begin{align} \label{eq:diff1}
  t(1-t) g''(t) + \big(2\mu + 2m+d - & \, (2\mu + \g+ 2m + d +1) t\big) g'(t) \\ 
        & = - (n-m)(n+m+2\mu+\g+d) g(t). \notag
\end{align}
Since $H(x,t)$ is a homogeneous polynomial of degree $m$ in $(x,t)$, it satisfies the identity
\begin{equation} \label{eq:diff2}
     \left(t \frac{\partial}{\partial t} + \la x, \nabla_x\ra\right) H(x,t) = m H(x,t). 
\end{equation}
Furthermore, let $h(x): = P_\kb^m(x)$; then $h\in \CV_m^d(\varpi_\mu)$ satisfies the differential 
equation \eqref{eq:diffBall}. Since $H(x,t) = t^m h(\frac{x}{t})$, it follows that 
\begin{align}\label{eq:diff3}
  \left( t^2 \Delta_x  - \la x, \nabla_x\ra^2  - (2 \mu +d-1) \la x,\nabla_x\ra \right) H 
      = -m (m+2\mu+d-1) H.
\end{align}

Now, taking derivatives, a straightforward computation gives
\begin{align*} 
  t(1-t) u_{tt} + 2(1-t) \la x,\nabla_x\ra u_t & = t(1-t) g''(t) H(x,t)   \\
  &   + 2 (1-t) g'(t) \left (t \frac{\partial}{\partial t} + \la x ,\nabla_x \ra\right)  H(x,t) \\
  &   +  (1-t) g(t) \left (2  \la x, \nabla_x\ra \frac{\partial}{\partial t}H(x,t) +t \frac{\partial^2}{\partial t^2}H(x,t) \right). 
\end{align*}
Using \eqref{eq:diff2} on the second term in the righthand side, we obtain from \eqref{eq:diff1}
\begin{align} \label{eq:diff4}
  t(1-t) u_{tt} + & 2(1-t) \la x,\nabla_x\ra u_t + \big(2 \mu+d - (2\mu+\g+d+1) t\big) u_t  \\
 &   =  - (n-m) (n+m+2\mu + \g+d) u - (\g+1) g(t)  t  \frac{\partial}{\partial t} H(x,t)   \notag \\
  &  +  (1-t) g(t) \left[ t \frac{\partial^2}{\partial t^2}H(x,t) + 
      (2  \la x, \nabla_x\ra +2 \mu + d)\frac{\partial}{\partial t}H(x,t)\right]. \notag
\end{align}
Since $ \frac{\partial}{\partial t} H(x,t)$ is a homogeneous polynomial of degree $m-1$ in $(x,t)$ variables, 
we deduce by \eqref{eq:diff2} that
$$
 t  \frac{\partial^2}{\partial t^2} H(x,t) = (m-1) \frac{\partial}{\partial t} H(x,t) - 
   \la x ,\nabla_x \ra \frac{\partial}{\partial t} H(x,t),
$$
so that $t$ times the square bracket in the righthand side of \eqref{eq:diff4} becomes, applying \eqref{eq:diff2} 
one more time, 
\begin{align*}
 t [\ldots ]  \, & = \big( 2\mu + m+d-1 + \la x \nabla_x \ra t  \big) t \frac{\partial}{\partial t} H(x,t) \\
       & =  \left( - \la x \nabla_x \ra^2 - (2\mu+d-1) \la x \nabla_x \ra + m (2\mu + m+d-1) \right)H(x,t) \\
       & =  - t^2 \Delta_x H(x,t),
 \end{align*}
where the last step follows from \eqref{eq:diff3}. With this identity and applying \eqref{eq:diff2} again in 
the second term in the righthand side of \eqref{eq:diff4}, we obtain 
\begin{align*}
 t(1-t) u_{tt} + & 2(1-t) \la x,\nabla_x\ra u_t + \big(2 \mu+d - (2\mu+\g+d+1) t\big) u_t + t (1-t) \Delta_x u\\   
   & = - (n-m) (n+m+2\mu + \g+d) u - m (\g+1) u + (\g+1) \la x, \nabla \ra u.
\end{align*}
Now, adding \eqref{eq:diff3} multiplied by $g(t)$ to the above identity and using 
$$
   (n-m) (n+m + 2\mu +\g + d) + m(m+2 \mu + \g+d) = n(n + 2\mu + \g +d),
$$
we conclude that 
\begin{align*}
  \big[ t(1-t)\partial_{tt}& +  2 (1-t) \la x,\nabla_x \ra \partial_t + t \Delta_x -  \la x,\nabla_x \ra^2
     + (2\mu+d)\partial_t  \\ 
  &  - (2\mu+\g+d+1) \big( \la x,\nabla_x\ra + t \partial_t\big) +\la x,\nabla_x\ra \big] u = 
        -n (n+2\mu+\g+d) u.
\end{align*}
Finally, using $\la x, \nabla_x \ra^2 = \la x, \nabla_x \ra + \sum_{1 \le i, j \le d} x_i x_j \partial_{x_i} \partial_{x_j}$, 
it is easy to verify that the lefthand side of the above identity is exactly $\CD_{\mu,\g}$. The proof is completed. 
\end{proof}

\begin{rem}
When $\b \ne 0$, the Jacobi polynomials $Q_{m,\kb}^n$ on the cone with respect to $W_{\mu,\b,\g}$ 
also satisfy a differential equation, but the eigenvalues depend on both $m$ and $n$. In other words,
$\CV_n(\VV^{d+1},W_{\mu,\b,\g})$ is not an eigenspace of such a differential operator. 
\end{rem}

\subsection{Lagurre polynomials on the cone}
 
We consider the weight functions 
$$
   W_{\mu,\b}^L(x,t): = (t^2-\|x\|^2)^{\mu-\f12} t^\b e^{-t}, \quad   \mu > -\tfrac12, \quad \b > -1.
$$
defined on the infinite solid cone 
$$
  \VV^{d+1} =\{(x,t): \|x\| \le t, \quad t \in (0, \infty), \, x \in \RR^d\}
$$ 
and define the inner product 
$$
   \la f,g \ra_{\mu,\b} : = b_{\mu,\b} \int_{\VV^{d+1}} f(x,t) g(x,t) W_{\mu,\b}^L(x,t) \d x\, \d t,
$$
where $b_{\mu,\b}$ is the constant chosen so that $\la 1, 1 \ra_{\mu,\b} =1$. Similar to \eqref{eq:int-V}, 
we have
\begin{align*}
  \int_{\VV^{d+1}} f(x,t) W_{\mu,\b}^L (x,t) \d x \, \d t \, = \int_0^\infty \int_{\BB^d} f(t u, t) t^{\b+2\mu+d-1} e^{-t} 
    \d t  (1-\|u\|^2)^{\mu -\f12} \d u, 
\end{align*}
which implies, in particular, that
$$
b_{\mu,\b} = c_{2\mu+\b+d-1}^L b_\mu^B \quad\hbox{with}\quad c_{\a}^L =   \frac{1}{\Gamma(\a+1)}.
$$
Let $\CV_n(\VV^{d+1}, W_{\mu,\b}^L)$ be the space of orthogonal polynomials of degree exactly $n$ with respect to 
the inner product $\la \cdot,\cdot\ra_{\mu,\b}$. Similar to the Jacobi case, a basis for the Laguerre polynomials on
the cone is given in the following: 

\begin{prop}\label{thm:OPconeL} 
For $m =0,1,2,\ldots$, let $\{P_\kb^m(\varpi_\mu): |\kb| = m, \kb \in \NN_0^d\}$ denote an orthonormal basis 
of $\CV_m^d(\varpi_\mu)$ on the unit ball $\BB^d$. Let $\a = \mu+ \f{\b+d-1}{2}$. Define
$$
  L_{m,\kb}^n(x,t) := L_{n-m}^{(2\a+2m)}(t) t^m P_\kb^m\left(\varpi_\mu; \frac{x}{t}\right).
$$
Then $\{L_{m,\kb}^n(x,t): |\kb| = m, \, 0 \le m\le n\}$ is an orthogonal basis of $\CV_n(\VV^{d+1}, W_{\mu,\b}^L)$ 
and the norm of $L_{m, \kb}^n$ is given by 
$$
    H_{m,n}:=  \la L_{m, \kb}^n, L_{m, \kb}^{n} \ra_{\mu,\b} = 
     \frac{ c_{2\a}^L} { c_{2\a+2m}^L} h_{n-m}^{(2\a+2m)},
$$
where $h_m^{(\a)}$ denotes the norm of the Laguerre polynomial of degree $m$. 
\end{prop}

When $\b = 0$, as in the Jacobi case, the Laguerre polynomials on the cone are also eigenfunctions of 
a second order linear PDE with eigenvalues depending only on the degree of the polynomials.
 
\begin{thm} \label{thm:DEconeL} 
Let $\mu > -\f12$ and $n \in \NN_0$. Then every $u \in \CV_n(\VV^{d+1}, W_{\mu,0}^L)$ satisfies the 
differential equation
\begin{equation}\label{eq:cone-eigenL}
   \CD_{\mu} u =  -n  u,
\end{equation}
where $\CD_{\mu} = \CD_{\mu}(x,t)$ is the second order linear differential operator
\begin{align*}
  \CD_{\mu} : = & \, t \left( \Delta_x+ \partial_t^2 \right) + 2  \la x,\nabla_x \ra \partial_t   - \la x,\nabla_x\ra 
        + (2\mu + d -t) \partial_t.
  \end{align*}
\end{thm}

\begin{proof}
The proof follows the same steps as in the Jacobi case, but simpler. We again write $u(x,t) = g(t)H(x,t)$, 
with $g(t) = L_{n-m}^{(2\mu+d-1+2m)}(t)$ and $H$ being the same as before. Instead of \eqref{eq:diff1}, we 
now use the differential equation for the Laguerre polynomials \cite[(5.1.2)]{Sz}, which shows that $g$ satisfies 
\begin{equation} \label{eq:diffL}
  t g''(t) +  (2\mu + 2m+d - t) g'(t)  = - (n-m) g(t). 
\end{equation}
As in the proof of Theorem \ref{thm:DEconeJ}, a straightforward computation shows that the above differential
equation for $g$ leads to an analog of \eqref{eq:diff4}, 
\begin{align} \label{eq:diffLag}
   t u_{tt} + 2 \la x,\nabla_x \ra u_t   & \, + (2\mu+d-t) u_t
    =  - (n-m) u -  g(t)  t  \frac{\partial}{\partial t} H(x,t)  \\
    & \, +   g(t) \left[ t \frac{\partial^2}{\partial t^2}H(x,t) + 
      (2  \la x, \nabla_x\ra +2 \mu + d)\frac{\partial}{\partial t}H(x,t)\right]. \notag
\end{align}
The term in the square bracket, being exactly the same as in the proof of Theorem \ref{thm:DEconeJ}, is equal 
to $- t \Delta_x H(x,t)$. Hence, using \eqref{eq:diff2} for $ t  \frac{\partial}{\partial t} H(x,t)$, we see that the
righthand side of \eqref{eq:diffLag} becomes 
$$
 -(n-m) u -  (m - \la x, \nabla_x \ra)u - t \Delta_x u = -n u + \la x ,\nabla_x \ra u - t \Delta_x u.
$$
Moving the last two terms to the lefthand side of \eqref{eq:diffLag}, we have proved \eqref{eq:cone-eigenL}.
\end{proof}

\begin{rem}
For $\b \ne 0$, the Laguerre polynomials on the cone also satisfy a differential equation, but the eigenvalues 
depend on both $m$ and $n$. 
In other words, $\CV_n(\VV^{d+1},W_{\mu,\b}^L)$ is not an eigenspace of such a differential operator. 
\end{rem}

\section{Reproducing kernel for the Jacobi polynomials on the cone}
\setcounter{equation}{0}

As the kernels of the orthogonal projection operators \eqref{eq:projOP}, reproducing kernels play an essential 
role for studying Fourier orthogonal expansions. In this section we describe a closed formula for the reproducing 
kernels of the Jacobi polynomials on the cone. 

In terms of the orthogonal basis in Proposition \ref{prop:OPconeJ}, the reproducing kernel of 
$\CV_n(\VV^{d+1},W_{\mu,\b,\g})$ satisfies 
\begin{equation} \label{eq:PnConeDef}
  \Pb_n(W_{\mu,\b,\g}; (x,t), (y,s)) =  \sum_{m=0}^n \sum_{|\kb| =m} \frac{Q_{m,\kb}^n(x,t) Q_{m,\kb}^n(y,s)}
    {H_{m,n}^{\a,\g}}. 
\end{equation}
To simplify our presentation, we consider $\b = 0$ and $\b \ne 0$ separately. 

\subsection{Jacobi polynomials on the cone with $\b =0$}
We consider the case $\b =0$, or the weight function $W_{\mu,\g}$ in \eqref{eq:W-mgCone} first. 
We begin with the reproducing kernel on the triangle $\VV^2$. 

For $d =1$, the cone $\VV^2$ becomes a triangle $\VV^2 = \{(x, t )\in \RR^2: |x| \le t \le 1\}$. It is related to 
the standard triangle $\TT^2$ by a change of variable $(x_1,x_2) \in \TT^2 \mapsto (x,t) \in \VV^2$,
\begin{equation*} 
    (x_1, x_2)  = \left(\frac{t+x}{2}, \frac{t-x}{2}\right).
\end{equation*}
Under this change of variables, the classical weight function $\varpi_{\a,\b,\g}$ on the triangle $\TT^2$ becomes 
$2^{-\a-\b} \wh \varpi_{\a,\b,\g}$, where 
$$
    \wh \varpi_{\a,\b,\g}(x,t): = (x+t)^\a (t-x)^\b (1-t)^\g,  \qquad (x,t) \in \VV^2.
$$
In particular, we see that $\wh \varpi_{\mu-\f12,\mu-\f12,\g} = W_{\mu,\g}$ when $d =1$. For $d \ge 2$, we 
need $\wh \varpi_{\a-\f12,\a - \f12,\g}$ for $\a = \mu + \f{d-1}{2}$. With a slight abuse of notation, 
we denote the reproducing kernel with respect to the weight function 
$$
\wh \varpi_{\a,\g}(x,t): = \wh \varpi_{\a-\f12,\a - \f12,\g}(x,t) = (t^2- x^2)^{\a-\f12}(1-t)^\g
$$
on the triangle $\VV^2$ by
\begin{align} \label{eq:PV-PT} 
\Pb_n \left(\wh \varpi_{\a,\g}; (x,t), (y,s) \right): = \Pb_n \left(\varpi_{\a-\f12,\a-\f12,\g}; (\tfrac{t+x}{2}, \tfrac{t-x}{2}),
  (\tfrac{s+y}{2}, \tfrac{s-y}{2}) \right), 
\end{align} 
where the righthand side is the reproducing kernel for $\CV_n(\TT^2,\varpi_{\a-\f12,\a-\f12,\g})$ on the triangle.
Recall that $Z_m^\a$ is defined in \eqref{eq:Zn}. 

\begin{lem} 
Let $\a = \mu + \f{d-1}{2}$. For $(u,t) \in \VV^2$ and $0 \le s \le 1$, 
\begin{align} \label{eq:trianglePV2} 
 \Pb_n\left(\wh \varpi_{\a,\g}; (u ,t), (s,s)\right)  = \sum_{m=0}^n \frac{h_m^{(\a- \f12, \a - \f12)}}
      {h_{m,n}^{(\a- \f12, \a - \f12,\g)}} & P_{n-m}^{(2m+2 \a,\g)}(1-2 t) \\
           \times & P_{n-m}^{(2m+2\a,\g)}(1-2s) t^m s^m Z_m^{\a}\left(\frac{u}{t}\right), \notag
\end{align}
where $h_{m}^{(\a,\b)}$ is the norm of the Jacobi polynomial in \eqref{eq:JacobiNorm} and 
$h_{m,n}^{\a,\b,\g}$ is the norm of the Jacobi polynomial on the triangle in \eqref{eq:triangleOPnorm}.
\end{lem}

\begin{proof}
By \eqref{eq:PV-PT} and \eqref{eq:trianglePn}, we see that the kernel can be written as
\begin{align} \label{eq:trianglePV} 
 \Pb_n\left(\wh \varpi_{\a,\g}; (u ,t), (v,s)\right) = & \sum_{k=0}^n 
    \frac{P_{n-k}^{(2k+2\a,\g)}(1-2 t)P_{n-k}^{(2k+2\a,\g)}(1-2s)}{h_{k,n}^{(\a-\f12,\a-\f12,\g)}} \\
       & \times t^k s^k P_k^{(\a-\f12,\a-\f12)} \left(\frac{u}{t}\right)
           P_k^{(\a-\f12,\a-\f12)} \left(\frac{v}{s}\right)  \notag
\end{align}
for $(u,t) \in \VV^2$ and $(v,s) \in \VV^2$, where $h_{k,n}^{(\mu-\f12,\mu-\f12,\g)}$ is given in 
\eqref{eq:triangleOPnorm}. 
The Gegenbauer polynomials are special cases of the Jacobi polynomials. In particular \cite[(4.7.1)]{Sz},
$$
  \frac{P_m^{(\a- \f12, \a - \f12)}(1)P_m^{(\a- \f12, \a - \f12)}(u)}{h_m^{(\a- \f12, \a - \f12)}} 
   = \frac{C_m^\a(1) C_m^\a(u)} {h_m^\a} = \frac{m+\a}{\a} C_m^{\a}(u) = Z_m^\a(u).
$$
Hence, setting $v = s$ in \eqref{eq:trianglePV}, we obtain \eqref{eq:trianglePV2}.  
\end{proof}

For $d \ge 2$, the reproducing kernel on the cone $\VV^{d+1}$ can be written as an integral of the 
reproducing kernel on $\VV^2$. 
 
\begin{thm} \label{thm:PnCone}
Let $d \ge 2$. For $\mu \ge 0$ and $\g > -1$, $(x,t) \in \VV^{d+1}$, $(y,s) \in \VV^{d+1}$, 
\begin{align}\label{eq:PbCone}
 & \Pb_n \big(W_{\mu,0,\g};  (x,t), (y,s)\big) \\
  &\qquad  = c_{\mu-\f12} \int_{-1}^1 \Pb_n  \big(\wh \varpi_{\a, \g}; (\zeta(x,t, y,s;u),t), (s,s) \big)
     (1-u^2)^{\mu-1} \d u, \notag
\end{align} 
where $\a = \mu + \frac{d-1}{2}$ and 
\begin{equation} \label{eq:zetaCone}
\zeta(x,t, y,s; u): = \frac{\la x,y \ra}{s} + \frac{\sqrt{t^2-\|x\|^2} \sqrt{s^2-\|y\|^2}}{s} u.
\end{equation}
In the case $\mu =  0$, the identity \eqref{eq:PbCone} holds under the limit \eqref{eq:limit-int}. 
\end{thm}

\begin{proof}
In terms of the orthogonal basis $Q_{m,\kb}^n$ in \eqref{eq:coneJ} of $\CV_n(\VV^{d+1})$, the reproducing 
kernel is given by 
\begin{align*}
  \Pb_n(W_{\mu,0,\g}; (x,t), (y,s))
   =  \sum_{m=0}^n & \frac{P_{n-m}^{(2\a+2m,\g)}(1-2t)P_{n-m}^{(2\a + 2m,\g)}(1-2s)} {H_{m,n}^{\mu,\g}}  \\
       & \times  s^m t^m P_m\left(\varpi_\mu; \frac{x}{t}, \frac{y}{s}\right).  
\end{align*}
Using the closed formula of the reproducing kernel  $P_m(\varpi_\mu; \cdot,\cdot)$ in \eqref{eq:PnBall} with
$$
   \left \langle \frac{x}{t}, \frac{y}{s} \right\rangle + \sqrt{ 1- \frac{\|x\|^2}{t^2} }\sqrt{ 1- \frac{\|y\|^2}{s^2} } u 
    =   \frac{\zeta(x,t, y,s; u)}{t}, 
$$
we see that
\begin{align}\label{eq:PbCone1.5}
  \Pb_n(W_{\mu,0,\g}; (x,t), (y,s))
   =    c_{\mu-\f12} & \int_{-1}^1 
     \sum_{m=0}^n  \frac{P_{n-m}^{(2\a+2m,\g)}(1-2t)P_{n-m}^{(2\a + 2m,\g)}(1-2s)} {H_{m,n}^{\mu,\g}}  t^m s^m  \notag \\
      & \times  Z_m^\a \left( \frac{\zeta(x,t,y,s; u)}{t}\right) (1-u^2)^{\mu-1}\d u.
\end{align}
The sum in the righthand side is comparable to the sum in the righthand side of \eqref{eq:trianglePV2}
since, using the explicit formulas for the constants in \eqref{eq:JacobiNorm}, \eqref{eq:triangleOPnorm} and \eqref{eq:coneJnorm}, we see that 
\begin{equation} \label{eq:norm=norm}
\frac{h_m^{(\a- \f12, \a - \f12)}}{h_{m,n}^{(\a- \f12, \a - \f12,\g)}}  = \frac{1}{H_{m,n}^{(\mu,\g)}} 
\qquad \hbox{when} \quad \a = \mu + \frac{d-1}{2}.
\end{equation}
Putting the two identities together, we have proved \eqref{eq:PbCone}.
\end{proof} 
  
The reproducing kernel on the triangle satisfies a closed formula, given in \eqref{eq:trianglePn2}, when 
the parameters are nonnegative. This allows us to derive a closed formula for 
the reproducing kernel on the cone from \eqref{eq:PbCone}. 
 
\begin{thm} \label{thm:PnCone2}
Let $d \ge 2$. For $\mu \ge 0$ and $\g \ge -\f12$, let $\a = \mu + \frac{d-1}{2}$; then 
\begin{align}\label{eq:PbCone2}
  \Pb_n \big(W_{\mu,0,\g}; (x,t), (y,s)\big) =\, & 
   c_{\mu-\f12}  c_{\a -\f{1}{2}} c_\g \int_{[-1,1]^3}  Z_{2n}^{2 \a+\g+1} (\xi (x, t, y, s; u, v)) \\
     &\times   (1-u^2)^{\mu-1} (1-v_1^2)^{\a - 1}(1-v_2^2)^{\g-\f12}  \d u \d v, \notag
\end{align}
where $\xi (x,t, y,s; u, v) \in [-1,1]$ is defined by 
\begin{align} \label{eq:xi}
\xi (x,t, y,s; u, v) = &\, v_1 \sqrt{\tfrac12 \left( s t + \la x,y \ra + \sqrt{t^2-\|x\|^2} \sqrt{s^2-\|y\|^2} \, u \right)}\\
      & + v_2 \sqrt{1-t}\sqrt{1-s}. \notag
\end{align}
When $\mu = 0$ or $\g = -\f12$, the identity \eqref{eq:PbCone2} holds under the limit \eqref{eq:limit-int}.
\end{thm}
 
\begin{proof}
By \eqref{eq:PV-PT}, we use \eqref{eq:trianglePn2} with $(\a, \b, \g) = (\mu + \f{d-2}2,\mu + \f{d-2}2,\g)$, 
$x_1 = (t+ \zeta)/2$, $x_2 = (t- \zeta)/2$, $y_1 = s$ and $y_2 =0$. Then, for $\eta$ defined by 
\eqref{eq:zetaT}, we have
\begin{align*} 
 \eta\left( (\tfrac{t+\zeta}{2}, \tfrac{t-\zeta}{2}), (s, 0), v) \right) = \sqrt{\tfrac{t+\zeta}{2} s} \, v_1 + \sqrt{1-t}\sqrt{1-s}\, v_2,    
\end{align*}
so that, by \eqref{eq:PV-PT},  
\begin{align} \label{eq:closedPV}
 & \Pb_n \left(\wh \varpi_{\a,\g}; (\zeta,t), (s,s) \right) = c_{\a-\f{1}{2}} c_\g  \\
  \times  \int_{-1}^1\int_{-1}^1 & Z_{2n}^{2 \a+\g+1} \left( \sqrt{\tfrac{t+\zeta}{2} s} \, v_1 + \sqrt{1-t}\sqrt{1-s}\, v_2\right)    
    (1-v_1^2)^{\a- 1}(1-v_2^2)^{\g-\f12}   \d v. \notag
\end{align}
Consequently, setting $\zeta = \zeta(x,t, y,s; u)$ defined in \eqref{eq:zetaCone}, we obtain
\eqref{eq:PbCone2} from \eqref{eq:PV-PT} and \eqref{eq:PbCone}. Finally, since $\|x\|\le t$ and $\|y\|\le s$, 
Cauchy's inequality shows that 
$$
  |\la x,y\ra + \sqrt{t^2-\|x\|^2} \sqrt{s^2-\|y\|^2} \, u| \le \|x\|\cdot\|y\|+  \sqrt{t^2-\|x\|^2} \sqrt{s^2-\|y\|^2} \le t s,
$$
form which it follows easily that $\xi (x,t, y,s; u, v) \in [-1,1]$. 
\end{proof}

In view of \eqref{eq:PnConeDef}, the identity \eqref{eq:PbCone2} can be regarded as an addition formula
for the Jacobi polynomials on the cone. 

\subsection{Jacobi polynomials on the cone with $\b > 0$} 
For $\b > 0$, the reproducing kernels of $\CV_n(\VV^{d+1}, W_{\mu,\b,\g})$ also satisfy a closed formula.
The formula will be more involved as it requires an additional layer of complication as can be seen in 
the next theorem and its proof. 
 
\begin{thm} \label{thm:PnConeG}
Let $d \ge 2$. For $\mu \ge 0$, $\b > 0$ and $\g > -1$, let $\a =  \mu + \frac{\b+d-1}{2}$. Then,
for $(x,t), (y,s) \in \VV^{d+1}$,  
\begin{align}\label{eq:PbConeG}
  \Pb_n \big(W_{\mu,\b,\g}; \, & (x,t), (y,s)\big)  = \wh c_{\mu,\b}  \int_{[-1,1]^3}
      \Pb_n \left(\wh \varpi_{\a, \g}; (\wh \zeta(x,t, y,s; z, u),t), (s,s) \right) \\
  &  \times  (1-u^2)^{\mu-1}  (1-z_1)^{\mu+\f{d-1}{2}} (1+z_1)^{\f{\b}2 -1} (1-z_2^2)^{\f{\b-1}{2}} \d u \d z, \notag
\end{align} 
where $\wh c_{\mu,\b} = c_{\mu-\f12} c_{\mu+\f{d-1}{2}, \f{\b-1}{2}}  c_{\frac{\b}{2}}$ and 
\begin{equation} \label{eq:hat-zetaCone}
\wh \zeta(x,t, y,s; z, u): = \frac{1-z_1}{2} \left( \frac{\la x,y \ra}{s} + \frac{\sqrt{t^2-\|x\|^2} \sqrt{s^2-\|y\|^2}}{s} u \right)
  +  \frac{1+z_1}{2} z_2t.
\end{equation}
In the case $\mu =  0$ or $\b = -\f12$, the identity \eqref{eq:PbConeG} holds under the limit \eqref{eq:limit-int}. 
\end{thm}

\begin{proof}
Following the proof of Theorem \ref{thm:PnCone}, we see that \eqref{eq:PbCone1.5} becomes 
\begin{align}\label{eq:PbCone1.5h}
  \Pb_n(W_{\mu,\b,\g}; (x,t), (y,s))
   =    c_{\mu-\f12} \int_{-1}^1 
     \sum_{m=0}^n & \frac{P_{n-m}^{(2\a+2m,\g)}(1-2t)P_{n-m}^{(2\a + 2m,\g)}(1-2s)} {H_{m,n}^{\a,\g}}  t^m s^m  \notag \\
      & \times  Z_m^{\mu + \f{d-1}{2}} \left( \frac{\zeta(x,t,y,s; u)}{t}\right) (1-u^2)^{\mu-1}\d u.
\end{align}
Notice that the index of $Z_m^{\mu + \f{d-1}{2}}$ is equal to $\a - \frac{\b}{2}$. In order to follow the proof of
Theorem \ref{thm:PnCone}, we need to increase this index to $\a$, which can be done by using the following 
identity, proved recently in \cite{X15},
\begin{align}\label{eq:ZtoZ}
  Z_m^\l(t) =  c_{\l,\s-1} c_{\s}  \int_{-1}^1 \int_{-1}^1 & Z_m^{\l+\s} \left( \tfrac{1-z_1}{2} t + \tfrac{1+z_1}2 z_2\right)\\
      & \times (1-z_1)^\l (1+z_1)^{\s-1}  (1-z_2^2)^{\s-\f12}  \d z \notag
\end{align}
with $\l = \mu+\frac{d-1}{2}$ and $\s = \frac{\b}{2}$, so that \eqref{eq:PbCone1.5h} becomes
\begin{align*}
  \Pb_n(W_{\mu,\b,\g}; (x,t), (y,s))
   =  \, &  \wh c_{\mu,\b} \int_{[-1,1]^3} 
       \sum_{m=0}^n  \frac{P_{n-m}^{(2\a+2m,\g)}(1-2t)P_{n-m}^{(2\a + 2m,\g)}(1-2s)}  {H_{m,n}^{\a,\g}}  \notag \\
      & \times    t^m s^m  Z_m^{\a} \left( \frac{1-z_1}{2}\frac{\zeta(x,t,y,s; u)}{t} + \frac{1+z_1}2 z_2\right) \\
      & \times (1-u^2)^{\mu-1}  (1-z_1)^\l (1+z_1)^{\s-1}  (1-z_2^2)^{\s-\f12}  \d z\, \d u.
\end{align*}
Comparing with \eqref{eq:PbCone1.5}, we see that the rest of the proof follows exactly as in the proof of 
Theorem \ref{thm:PnCone}. 
\end{proof}

We also have an analogue of Theorem \ref{thm:PnCone2} that shows $\Pb_n \big(W_{\mu,\b, \g})$ also
possesses a structure of one-dimension. 

\begin{thm} \label{thm:PbCone2G}
Let $d \ge 2$. For $\mu \ge 0$, $\b > 0$ and $\g \ge -\f12$, let $\a = \mu + \frac{\b}{2}+ \frac{d-1}{2}$. Then 
\begin{align}\label{eq:PbConeG2}
  \Pb_n \big(W_{\mu,\b, \g}; (x,t), (y,s)\big) =\, & 
   \wh c_{\mu,\b}  c_{\a -\f{1}{2}} c_\g \int_{[-1,1]^5}  Z_{2n}^{2 \a+\g+1} (\wh \xi (x, t, y, s; z, u, v)) \\
     &\times  (1-z_1)^{\mu+\f{d-1}{2}} (1+z_1)^{\f{\b}2 -1} (1-z_2^2)^{\f{\b-1}{2}}  \notag \\  
     & \times (1-u^2)^{\mu-1} du (1-v_1^2)^{\a - 1}(1-v_2^2)^{\g-\f12} \d z \,\d u\, \d v, \notag
\end{align}
where $\wh \xi (x,t, y,s; u, v) \in [-1,1]$ is defined by 
\begin{align} \label{eq:hatxi}
& \wh \xi (x,t, y,s; u, v)  =  v_2 \sqrt{1-t}\sqrt{1-s} \\
     & \quad +\frac12 v_1 \sqrt{2 s t + (1-z_1) \left( \la x,y \ra + \sqrt{t^2-\|x\|^2} \sqrt{s^2-\|y\|^2} \, u \right) 
   +  (1+z_1)z_2 st}. \notag
\end{align}
In the case $\mu =  0$ or $\g = -\f12$, the identity \eqref{eq:PbConeG} holds under the 
limit \eqref{eq:limit-int}. 
\end{thm}

\begin{proof}
The proof can be carried out as that of Theorem \ref{thm:PnConeG}. By \eqref{eq:PV-PT}, we again 
use \eqref{eq:trianglePn2} but with $(\a, \b, \g) = (\mu + \f{\b+d-2}2,\mu + \f{\b+d-2}2,\g)$, 
$x_1 = (t+ \wh \zeta)/2$, $x_2 = (t- \wh \zeta)/2$, $y_1 = s$ and $y_2 =s$, where $\wh \eta$ is defined by 
\eqref{eq:hat-zetaCone}. This leads to the formula $\Pb_n \left(\wh \varpi_{\a,\g}; (\zeta,t), (s,s) \right)$
given in \eqref{eq:closedPV} with $\zeta$ replaced by $\wh \zeta$ and $\a = \mu + \frac{\b + d-1}{2}$. 
Substituting it into \eqref{eq:PbConeG}, we then have \eqref{eq:PbConeG2}. 
\end{proof}

Although the formula \eqref{eq:PbConeG2} is fairly complicated, what is of important is that it shows
that the kernel has a one-dimensional structure. We will make use of this structure in the following section. 

\section{Convolution structure on the cone}
\setcounter{equation}{0}
 
The closed formula for the reproducing kernels suggests a convolution structure on the cone that is useful
in the study of Fourier series in the Jacobi polynomials on the cone. For this development, what is important
is the one-dimensional structure of the kernels, not the explicit formula of the closed formula itself.  

We start with a definition suggested by the closed formula \eqref{eq:PbConeG2}.

\begin{defn}\label{defn:Tmg} 
Let $d \ge 2$. For $\mu \ge 0$, $\b \ge 0$ and $\g \ge -\f12$, define $\a = \mu + \frac{\b+d-1}{2}$. 
For $g\in L^1([-1,1],w_{2\a+\g+1})$, we define the operator $T_{\mu,\b, \g}$ on the cone $\VV^{d+1}$ by
\begin{align} \label{eq:Tmg2}
   T_{\mu,\b, \g} g\big((x,t),(y,s)\big) := & \wh c_{\mu,\b}  c_{\a -\f{1}{2}} c_\g \int_{[-1,1]^5}  
     g \left( \wh \xi (x, t, y, s; z, u, v)\right) \\
     &\times  (1-z_1)^{\mu+\f{d-1}{2}} (1+z_1)^{\f{\b}2 -1} (1-z_2^2)^{\f{\b-1}{2}}  \notag \\  
     & \times (1-u^2)^{\mu-1} du (1-v_1^2)^{\a - 1}(1-v_2^2)^{\g-\f12} \d z \,\d u\, \d v, \notag
\end{align}
where $\wh \xi(x,t,y,s;u,v)$ is defined by \eqref{eq:hatxi}. When $\mu = 0$ or $\b = 0$ or $\g = -\f12$, 
the definition holds under the limit \eqref{eq:limit-int}. If $\b = 0$, the definition is simplified to 
\begin{align} \label{eq:Tmg}
   T_{\mu,0,\g} g\big((x,t),(y,s)\big) := & c_{\mu-\f12}  c_{\a-\f{1}{2}} c_\g \int_{[-1,1]^3}  g \left( \xi (x, t, y, s; u, v)\right) \\
     &\times   (1-u^2)^{\mu-1} du (1-v_1^2)^{\a-1}(1-v_2^2)^{\g-\f12}  \d u \d v, \notag 
\end{align}
where $\xi(x,t,y,s;u,v)$ is defined by \eqref{eq:xi}. 
\end{defn}

By the closed formulas of the reproducing kernel \eqref{eq:PbCone2} and \eqref{eq:PbConeG2}, we immediately
obtain
\begin{equation} \label{eq:PbZn}
   \Pb_n\big(W_{\mu,\b,\g}; (x,t), (y,s)\big) = T_{\mu,\b,\g} Z_{2n}^{2\a + \g+ 1} \big((x,t),(y,s)\big),
\end{equation}
which is our motivation for the definition of $T_{\mu,\b,\g}$. 

In the following, we consider the $L^p$ norm $\|\cdot\|_{L^p}$ for $1\le p \le \infty$. We will always assume that
the case $p = \infty$ is the uniform norm over continuous functions. Because of \eqref{eq:PbZn} and the 
fact that $Z_{2n}^{\l}$ is an even polynomial, we only need the action of $T_{\mu,\b,\g}$ on the function $g$ that
is even for the purpose of studying the Fourier orthogonal series.  

\begin{lem} \label{lem:translateT}
Let $g \in L^1([-1,1],w_{2\a+\g+1})$ be an even function on $[-1,1]$. Then 
\begin{enumerate}
\item for each $Q_n \in \CV_n (\VV^{d+1}, W_{\mu,\b,\g})$, 
\begin{equation}\label{eq:FHcone}
  b_{\mu,\b,\g} \int_{\VV^{d+1}}  T_{\mu,\b,\g} g\big((x,t),(y,s)\big) Q_n (y,s) W_{\mu,\b,\g}(y,s) \d y \d s 
        =  \Lambda_n (g) Q_n(x,t),
\end{equation}
where $b_{\mu,\b,\g}$ is the normalization constant of $ W_{\mu,\b,\g}$ and
$$
 \Lambda_n (g) = c_{2\a+\g+1} \int_{-1}^1 g(t) \frac{C_{2n}^{2\a+\g+1} (t)} {C_{2n}^{2\a+\g+1}(1)}
      (1-t^2)^{2 \a+\g + \f12} \d t.
$$
\item for $1\le p \le \infty$ and $(x,t) \in \VV^{d+1}$, 
\begin{equation}\label{eq:Tbound}
  \left \| T_{\mu,\b,\g} g\big((x,t),(\cdot,\cdot)\big) \right\|_{L^p(\VV^{d+1}, W_{\mu,\b,\g})}
   \le \|g \|_{L^p([-1,1],w_{2\a + \g+1})}.
\end{equation}
\end{enumerate}
\end{lem}

\begin{proof}
If $g$ is an even polynomial of degree at most $n$, then we can expand it in terms of the Gegenbauer polynomials
of even degrees. That is, we can write
$$
g(t) = \sum_{m=0}^n \Lambda_k Z_{2k}^{2\a+\g+1} (t),
$$ 
where, by orthogonality and the fact that $h_k^\l = \frac{\l}{k+\l} C_k^\l(1)$,
$$
   \Lambda_k =  \frac{c_{2\a+\g+1}}{C_{2k}^{2\a+\g+1}(1)} \int_{-1}^1 g(t) C_{2k}^{2\a+\g+1} (t) 
      (1-t^2)^{2 \a+\g+ \f12} \d t.
$$
Using the formula \eqref{eq:PbZn}, we have
$$
  T_{\mu,\b,\g} g\big((x,t),(y,s)\big) = \sum_{k=0}^n \Lambda_k \Pb_k \big(W_{\mu,\b,\g}; (x,t),(y,s) \big). 
$$
Consequently, by the reproducing property, for each $Q_k \in \CV_n(\VV^{d+1}, W_{\mu,\b,\g})$, $k \le n$, we
conclude  
$$
  b_{\mu,\b,\g}\int_{\VV^{d+1}}  T_{\mu,\b,\g} g\big((x,t),(y,s)\big) Q(y,s) W_{\mu,\b,\g}(y,s) \d y \d s 
        = \Lambda_k Q_k(x,t). 
$$
This establishes the lemma when $g$ is a polynomial. The usual density argument then completes the 
proof of (1). 

For $p =\infty$, the inequality \eqref{eq:Tbound} is evident. If $g$ is nonnegative, then $T_{\mu,\g} g$ is evidently 
nonnegative, so that $|T_{\mu,\b,\g} g| \le T_{\mu,\b,\g}(|g|)$. Hence, applying (1) with $n=0$ we see that
the inequality \eqref{eq:Tbound} holds for $p =1$. The case $1 < p < \infty$ follows from the Riesz-Thorin theorem.
\end{proof}
 
For $(x,t) \in \VV^{d+1}$, the operator $g \mapsto T_{\mu,\b,\g} g \big((x,t), (\cdot,\cdot)\big)$ defines a 
``translation" of $g$ by $(x,t)$. We use this operator to define a convolution structure with respect to 
$W_{\mu,\b,\g}$ on the cone. 

\begin{defn}\label{defn:convolCone}
Let $\mu \ge 0$, $\b \ge 0$ and $\g \ge -\f12$, and let $\a = \mu + \f{\b}{2} + \frac{d-1}{2}$. 
For $f \in L^1(\VV^{d+1}, W_{\mu,\b,\g})$ and $g \in L^1([-1,1],w_{2\a+\g+1})$, define the convolution 
of $f$ and $g$ on the cone by 
$$
  (f \ast_{\mu,\b, \g} g)(x,t)  := b_{\mu,\b,\g} \int_{\VV^{d+1}} f(y,s) 
      T_{\mu,\b,\g} g \big((x,t),(y,s)\big) W_{\mu,\b, \g}(y,s) \d y \d s.
$$
\end{defn}

The convolution on the cone satisfies Young's inequality:

\begin{thm} \label{thm:Young}
Let $p,q,r \ge 1$ and $p^{-1} = r^{-1}+q^{-1}-1$. For $f \in L^q\left( \VV^{d+1}, W_{\mu,\b,\g}\right)$ and
$g \in L^r(w_{2\a+\g+1};[-1,1])$ with $g$ an even function, 
\begin{equation} \label{eq:Young}
  \|f \ast_{\mu,\b,\g} g\|_{L^p\left( \VV^{d+1}, W_{\mu,\b,\g} \right)} \le \|f\|_{L^q\left( \VV^{d+1}, W_{\mu,\b,\g}\right)}
                   \|g\|_{L^r([-1,1],w_{2\a+\g+1})}.
\end{equation}
\end{thm}

\begin{proof}
The standard proof applies in this setting. By Minkowski's inequality,
\begin{align*}
    &      \|f  \ast_{\mu,\b,\g} g\|_{L^r \left( \VV^{d+1}, W_{\mu,\b,\g} \right)} \le  b_{\mu,\b,\g} \int_{\VV^{d+1}} |f(y,s)| \\
   & \qquad \times \left(b_{\mu,\b,\g}  \int_{\VV^{d+1}} \left| T_{\mu,\b,\g} g\big((x,t),(y,s)\big) \right|^r
        W_{\mu,\b,\g}(y,s) \d y \d s \right)^{1/r} \d x \d t.
\end{align*}
By \eqref{eq:Tbound} in Lemma \ref{lem:translateT}, we then conclude that 
$$
  \|f  \ast_{\mu,\b,\g} g\|_{L^r \left( \VV^{d+1}, W_{\mu,\b,\g} \right)} \le  \|f \|_{L^1 \left( \VV^{d+1}, W_{\mu,\b,\g} \right)} 
       \|g\|_{L^r([-1,1],w_{2\a+\g+1})}. 
$$
Furthermore, by H\"older's inequality and \eqref{eq:Tbound}, we see that 
$$
\|f \ast_{\mu,\b,\g} g\|_\infty \le  \|f \|_{L^{r'} \left( \VV^{d+1}, W_{\mu,\b,\g} \right)}  \|g\|_{L^r([-1,1],w_{2\a+\g+1})}, 
$$
where $\frac{1}{r'} + \frac1{r} = 1$. The inequality \eqref{eq:Young} follows from interpolating the 
above two inequalities with $\t = r(1-\frac1{p})$ by the Riesz-Thorin theorem. 
\end{proof}

The next proposition justifies calling $f*g$ convolution. Let $\wh g_n^\l$ be the Fourier--Gegenbauer 
series of $g$ defined by 
$$
   \wh g_n^\l = c_\l \int_{-1}^1 g(u) \frac{C_n^\l(u)}{C_n^\l(1)} (1-u^2)^{\l-\f12} du. 
$$
The projection operator $\proj_n^{\g,\b,\mu}: L^2(\VV^{d+1},W_{\mu,\b,\g})\mapsto \CV_n(\VV^{d+1},W_{\mu,\b,\g})$
is defined by 
\begin{equation}\label{eq:projCone}
   \proj_n^{\mu,\b,\g} f(x,t) = b_{\mu,\b,\g} \int_{\VV^{d+1}} f(y,s) \Pb_n\big(W_{\mu,\b,\g}; (x,t), (y,s) \big ) 
      W_{\mu,\b,\g}(y,s) \d y \d s.
\end{equation}

\begin{prop} \label{prop:proj-conv}
For $f \in L^1(\VV^{d+1}, W_{\mu,\b,\g})$ and $g \in L^1([-1,1],w_{2\mu+\g+d})$,
$$
   \proj_n^{\mu,\b,\g} (f \ast_{\mu,\b,\g} g) =  \wh g_n^{2\mu+\b+\g+d} \proj_n^{\mu,\b,\g} f.
$$
\end{prop}

\begin{proof}
For each $(x,t)$,  $\Pb_n\big(W_{\mu,\g}; (x,t), (\cdot,\cdot)\big)$ is an element of $\CV_n(\VV^{d+1}; W_{\mu,\g})$.
Hence, by the identity \eqref{eq:FHcone} and the definition of $f\ast_{\mu,\g} g$, we obtain
\begin{align*}
\proj_n^{\mu,\b,\g}( f\ast_{\mu,\b,\g} g)(x,t)&= b_{\mu,\b,\g} \int_{\VV^{d+1}} (f\ast_{\mu,\b,\g} g) (y,s) \\
     & \qquad \times \Pb_n\big(W_{\mu,\b,\g}; (x,t), (y,s)\big) W_{\mu,\b,\g}(y,s) \d y\d s\\ 
  & =  \wh g_n^{2\mu+\b+\g+d}  b_{\mu,\b,\g} \int_{\VV^{d+1}} f(u,r) \Pb_n\big(W_{\mu,\b,\g}; (x,t), (u, r)\big) \d u \d r \\
  & =  \wh g_n^{2\mu+\b+\g+d} \proj_n^{\mu,\b,\g} f(x,t),
  \end{align*}
where we have used the Fubini theorem in the second step. 
\end{proof}

By \eqref{eq:PbCone2} and the definition of the convolution operator, we have the following: 

\begin{prop}
For $f \in L^1(\VV^{d+1}, W_{\mu,\b,\g})$, 
$$
  \proj_n^{\mu,\b,\g} f = f  \ast_{\mu,\b,\g} Z_{2n}^{2\mu+\b+\g+d}. 
$$
\end{prop}

This shows that the convolution structure we defined can be used to study the Fourier series in the Jacobi
polynomials on the cone, which we shall explore in the next section. 

\section{Fourier series in the Jacobi polynomials on the cone}
\setcounter{equation}{0}

We now consider the Fourier orthogonal expansion with respect to the orthogonal polynomials on the cone.

The $n$-th partial sum operator $S_n(W_{\mu,\b, \g};f)$ is defined by 
$$
      S_n(W_{\mu,\b,\g};f) = \sum_{k=0}^n \proj_k^{\mu,\b,\g} f,
$$
which is the least square polynomial of degree $n$ in $L^2(\VV^{d+1}, W_{\mu,\b,\g})$. Evidently, this 
operator can be written as an integral operator,
$$
  S_n(W_{\mu,\b,\g};f) = b_{\mu,\b,\g} \int_{\VV^{d+1}} f(y,s) \Kb_n\big(W_{\mu,\b,\g}; (x,t), (y,s) \big)
     W_{\mu,\b,\g}(y,s)\d y \d s,
$$
where the kernel $\Kb_n\big(W_{\mu,\b,\g})$ is given by
$$
  \Kb_n\big(W_{\mu,\b,\g}; (x,t), (y,s) \big) = \sum_{k=0}^n \Pb_k\big(W_{\mu,\b,\g}; (x,t), (y,s) \big).
$$

We first show that this operator can be written in terms of the $n$-th partial sum of the Fourier--Jacobi series.
For $\a,\b > -1$, and $f\in L^2 (w_{\a,\b},[-1,1])$, let $s_n(w_{\a,\b}; f)$ denote the partial sum of the 
Fourier--Jacobi series defined by
$$
  s_n(w_{\a,\b}; f, u) :=  c_{\a,\b} \int_{-1}^1 f(v) k_n(w_{\a,\b}; u,v) w_{\a,\b}(v)\d v, 
$$
where the kernel $k_n(w_{\a,\b})$ is defined by 
$$
  k_n(w_{\a,\b}; u,v) =   \sum_{k=0}^n \frac{P_k^{(\a,\b)}(u)P_k^{(\a,\b)}(v)}{h_k^{(\a,\b)}}.
$$

\begin{prop} 
Let $\mu \ge 0$,  $\b \ge0$ and $\g \ge -\f12$, and let $\a = \mu+ \f{\b+ d-1}{2}$. For $(x,t), (y,s) \in \VV^{d+1}$, 
\begin{align} \label{eq:KnCone}
  & \Kb_n \left(W_{\mu,\b,\g}; (x,t),(y,s)\right) \\
     & \qquad \qquad
        = T_{\mu,\b, \g} \left[k_n \big(w_{2\a + \g + \f12, -\f12}; 2\{ \cdot\}^2-1, 1\big)\right ]\big((x,t),(y,s)\big).  \notag
\end{align}
Furthermore, for $(y,s) \in \VV^{d+1}$, 
\begin{align} \label{eq:KnCone2}
   \Kb_n \big(W_{\mu,\b, \g}; (0,0),(y,s)\big)  = k_n \big(w_{2\a,\g}; 1-2s, 1\big).
\end{align}
\end{prop}

\begin{proof}
The kernel $\Kb_n(W_{\mu,\b,\g})$ is a sum over $k$ of $\Pb_k\big(W_{\mu,\b,\g}\big)$. By \eqref{eq:PbZn}, 
this requires summing over $Z_{2k}^{2\a + \g+1}$. The subindex $2k$ is undesirable for the sum. We use 
instead the relation
$$
 Z_{2k}^\l (u) = \frac{2k+\l}{\l}C_{2k}^\l(u) = \frac {P_k^{(\l-\f12, -\f12)}(1)P_k^{(\l-\f12, -\f12)}(2u^2-1)}{h_k^{(\l-\f12, -\f12)}}
$$
that follows form the quadratic transform \eqref{eq:Jacobi-Gegen}. This allows us to sum over $k$ for $
0 \le k\le n$ and leads to \eqref{eq:KnCone}. 
Furthermore, by \eqref{eq:hatxi}, $\xi(0,0,y,s;u,v) = v_2 \sqrt{1-s}$, so that  
\begin{align*}
  \Kb_n \left(W_{\mu,\b,\g}; (0,0),(y,s)\right) = c_\g \int_{-1}^1 k_n \left(w_{2\a + \g + \f12, -\f12}; 
      2(1-s)v^2-1, 1\right) (1-v^2)^{\g - \f12} dv
\end{align*}
by the definition of $T_{\mu,\b,\g}$. Hence, \eqref{eq:KnCone2} follows from the identity
$$
 c_{\tau-\f12} \int_{-1}^1 \frac{P_k^{(a,b)}(1)P_k^{(a,b)}(2 s u^2 -1)}{h_k^{(a,b)}} (1-u^2)^{\tau-1} \d u
     =  \frac{P_k^{(a-\tau, b+ \tau)}(1)P_k^{(a-\tau,b+\tau)}(1-2s)}{h_k^{(a-\tau,b+ \tau)}} 
$$
with $a = 2\a+\g+\f12$, $b = -\f12$ and $\tau = \g+\f12$. The identity is an equivalent form of the 
Dirichlet-Mehler formula \cite[(4.10.12)]{Sz} for the Jacobi polynomials. 
\end{proof}

\begin{cor}
For $\mu \ge 0$, $\g \ge -\f12$ and $\a = \mu + \f{\b}2+\f{d-1}{2}$, 
\begin{equation} \label{eq:SnCone}
   S_n \left(W_{\mu,\b,\g}; f\right) = f \ast_{\mu,\b,\g} k_n \big(w_{2\a + \g + \f12, -\f12}; 2\{ \cdot\}^2-1, 1\big). 
\end{equation}
\end{cor}

The partial sum operator may not converge in $L^p$ norm for $p \ne 2$, so we may need to study summability
methods for the Fourier orthogonal expansions. The above corollary shows that the Fourier series in the Jacobi 
polynomials on the cone has a one-dimensional structure in terms of the Fourier--Jacobi series, from which
its properties could be derived accordingly. We consider the Ces\`aro $(C,\delta)$ means as an example. 

For $\delta > 0$, the Ces\`aro $(C,\delta)$ means $S_n^\delta (W_{\mu,\b,\g};f)$ of the Fourier series in the Jacobi
polynomials on the cone is defined by 
$$
 S_n^\delta (W_{\mu,\b,\g};f) := \f{1}{\binom{n+\delta}{n}} \sum_{k=0}^n \binom{n-k+\delta}{n-k} \proj_k^{\mu,\b,\g} f,
$$
which can be written as an integral of $f$ against the kernel  $K_n^\delta (W_{\mu,\g}; \cdot,\cdot)$ defined by
$$
\Kb_n^\delta \big(W_{\mu,\b,\g}; (x,t), (y,s)\big) :=  \f{1}{\binom{n+\delta}{n}} \sum_{k=0}^n \binom{n-k+\delta}{n-k} 
    \Pb_k\big(W_{\mu,\b,\g};(x,t),(y,s)\big). 
$$
Likewise, we denote by $k_n^\delta (w_{\a,\b}; \cdot, \cdot)$ the kernel for the Ces\`aro  $(C,\delta)$ means 
of the Fourier--Jacobi series, 
$$
  k_n^\delta (w_{\a,\b}; u,v) = \frac{1}{\binom{n+\delta}{n}} \sum_{k=0}^n \binom{n-k+\delta}{n-k} 
        \frac{P_k^{(\a,\b)}(u)P_k^{(\a,\b)}(v)}{h_k^{(\a,\b)}}.
$$
 
\begin{thm}
For $\mu \ge 0$ and $\g\ge -\f12$, define $\l_{\mu,\b,\g}: = 2\mu +\b+\g+d$. Then, the Ces\`aro $(C,\delta)$ 
means for $W_{\mu,\b,\g}$ on $\VV^{d+1}$ satisfy 
\begin{enumerate} [\quad 1.]
\item if $\delta \ge \l_{\mu,\b,\g} + 1$, then $S_n^\delta(W_{\mu,\b,\g}; f)$ is nonnegative if $f$ is nonnegative;
\item $S_n^\delta (W_{\mu,\b,\g}; f)$ converge to $f$ in $L^1(\VV^{d+1}, W_{\mu,\b,\g})$ norm or $C(\VV^{d+1})$ 
norm if $\delta > \l_{\mu,\b,\g}$ and only if $\delta > \l_{\mu,\b,\g}$ when $\g = -\f12$. 
\end{enumerate}
\end{thm}

\begin{proof}
Recall that $\a = \mu + \frac{\b}{2} + \frac{d-1}{2}$, so that $\l_{\mu,\b,\g} = 2 \a + \g +1$. From \eqref{eq:KnCone}, it follows immediately that 
$$
   \Kb_n^\delta \left(W_{\mu,\b,\g}; (x,t),(y,s)\right) = 
      T_{\mu,\g} \left[k_n^\delta \big(w_{2\a + \g + \f12, -\f12}; 2\{ \cdot\}^2-1, 1\big)\right ]\big((x,t),(y,s)\big).
$$
Hence, the first assertion follows immediately from the fact  \cite{Gas} that $k_n^{\delta}(w_{a,b}; u,v) \ge 0$ 
if $\delta \ge a+b+2$, which is $\delta\ge 2 \a + \g +2 = \l_{\mu,\b,\g} +1$ with $a =2\a + \g + \f12$ and $b = -\f12$. 

To prove the convergence of the second assertion, it is sufficient to show that 
$$
  \max_{(x,t)}\int_{\VV^{d+1}}   \left | \Kb_n^\delta(W_{\mu,\b,\g}; (x,t),(y,s))\right | 
    W_{\mu,\b,\g}(y,s)\d y \d s 
$$
is bounded. By \eqref{eq:Tbound}, we see that this quantity is bounded by 
\begin{align*}
 &    \int_{-1}^1 \left | k_n^\delta \big(w_{2\a + \g + \f12, -\f12}; 2 s^2-1, 1\big) \right| 
   (1-s^2)^{2\a+\g+1} \d s \\
  &  =   \frac{1}{2^{2\a+\g+2}} \int_{-1}^1 \left | k_n^\delta \big(w_{2\a + \g + \f12, -\f12}; u, 1\big) \right| 
   (1-u)^{2\a+\g+ \f12} (1+u)^{-\f12} \d t,
\end{align*}
where we have written the first integral as over $[0,1]$, since the integrant is even, and then change
the variable $u = 2s^2-1$. That this term is bounded for $\delta > 2\a + \g+1$ follows from the fact
that $s_n^\delta (w_{a,b}; f)$ converges to $f$ in the uniform norm if and only if $\delta >  \max\{a,b\}+\f12$.   
Finally, using \eqref{eq:KnCone} and \eqref{eq:int-V}, we see that 
\begin{align*}
&\int_{\VV^{d+1}} \left | \Kb_n^\delta(W_{\mu,\b,\g}; (0,0),(y,s))\right | W_{\mu,\b,\g}(y,s)\d y \d s \\
& \quad  =  \frac1{b_\mu } \int_0^1  \left |k_n^\delta \big(w_{2\a, \g}; 1-2s, 1\big) \right| s^{2\a}(1-s)^\g \d t,
\end{align*}
which is bounded if and only if $\delta > 2 \a+\f12$ (\cite[Theorem 9.1.3]{Sz}), whereas when $\g = -\f12$, 
$\l_{\mu,\b,\g} = 2 \a + \f12$. This completes the proof. 
\end{proof}
 
\section{Orthogonal structure on the surface of a cone}
\setcounter{equation}{0}

We now consider orthogonal polynomials on the surface of the cone $\VV^{d+1}$, which we denote by
$$
\VV_0^{d+1} := \{(x,t) \in \VV^{d+1}: \|x\|= |t|, \, 0 \le t \le b\},
$$
with respect to a bilinear form defined by
\begin{align}\label{eq:innerSF}
   \la f,g \ra_w : = b_w \int_{\VV_0^{d+1}} f(x,t) g(x,t) w(t) d \sigma(x,t),
\end{align}
where $d\sigma(x,t)$ is the Lebesgue measure on $\VV_0^{d+1}$, $w$ is a nonnegative function defined
on $\RR_+$ such that $\int_{\RR} t^{d-1} w(t) dt < \infty$ and $b_w$ is a normalized constant so that $\la 1,1\ra =1$. 

The bilinear form $\la \cdot,\cdot\ra_w$ is an inner product on the space $\RR[x,t] / \la \|x\|^2 - t^2\ra$, where
$\RR[x,t] / \la p\ra$ denotes the space of polynomials in $(x,t)$ variables modulo the polynomial idea $\la p\ra$
generated
by the polynomial $p$. Let $\CV_n(\VV_0^{d+1}, w)$ be the space of 
orthogonal polynomials with respect to the inner product $\la \cdot,\cdot \ra_w$. Since $\|x\|^2-t$ is a quadratic 
polynomial, it is not difficult to see that 
$$
   \dim \CV_n(\VV_0^{d+1},w)  = \binom{n+d-1}{n}+\binom{n+d-2}{n-1},
$$
the same as the dimension of the space of spherical polynomials of degree $n$ on $\SS^d$. 

As in the case of the cone, we consider two families of $w$, the Jacobi weight and the Laguerre
weight. The notation for some normalization constants may overlap with those already used in the 
previous sections, but they should cause little confusion since the values of these constants are of
little substance. 

\subsection{Jacobi polynomials on the surface of the cone}
For $d \ge 2$, we consider the cone with $b =1$ and choose $w$ as the Jacobi weight 
$$
     \varphi_{\b,\g}(t) = t^\b(1-t)^\g, \qquad \b > -d, \, \g > -1.
$$ 
We then define the inner product on $\RR[x,t] / \la \|x\|^2 - t^2\ra$ in terms of $w= \varphi_{\b,\g}$ by 
$$
    \la f,g\ra_{\b,\g} =  b_{\b,\g} \int_{\VV_0^{d+1}} f(x,t) g(x,t)  t^\b(1-t)^\g \d \sigma(x,t).
$$
The integral on the surface of the cone can be written as 
\begin{align}\label{eq:intVsf}
  \int_{\VV_0^{d+1}} f(x,t) d\sigma(x,t) \, & = \int_0^1 \int_{\|x\| = t} f(x,t) \d \sigma(x, t)   \\
    & = \int_0^1 t^{d-1} \int_{\sph} f(t\xi, t) d\sigma(\xi)  \d t, \notag
\end{align}
which gives, in particular, that 
$$
   b_{\b,\g} =  \frac{1}{\o_d} \frac{1}{\int_0^1 t^{\b+d -1}(1-t)^\g \d t} = \frac{1}{\o_d} c_{\b+d-1,\g},
$$
where $\o_d$ is the surface area of $\sph$ and $c_{\a,\g}$ is defined in \eqref{eq:c_ab}.

A basis of $\CV_n(\VV_0^{d+1}, \varphi_{\b,\g})$ can be given in terms of the Jacobi polynomials
and spherical harmonics. Let $\{Y_\ell^m: 1 \le \ell \le \dim \CH_m^d\}$ denote an orthonormal basis of 
$\CH_m^d$. We define
\begin{equation} \label{eq:sfOPbasis}
  S_{m, \ell}^n (x,t) = P_{n-m}^{(2m + \b + d-1,\g)} (1-2t) Y_\ell^m (x), \quad 0 \le m \le n, \,\, 1 \le \ell \le \dim \CH_m^d.
\end{equation}
Evidently, each $S_{m,\ell}^n$ is a polynomial of degree $n$ in $(x,t)$. 

\begin{prop} 
For $\b > - d$ and $\g > -1$, the set $\{S_{m,\ell}^n: 0\le m \le n, \, 1 \le \ell \le \dim \CH_m^d\}$ is an orthogonal
basis of $\CV_n(\VV_0^{d+1}, \varphi_{\b,\g})$. More precisely, 
\begin{equation} \label{eq:OPsfJ}
\la S_{m, \ell}^n, S_{m', \ell'}^{n'} \ra_{\b,\g} = H_{m,n}^{\b,\g} \delta_{n,n'} \delta_{m,m'} \delta_{\ell,\ell'},
\end{equation}
where, with $\a = \frac{\b+d-1}{2}$, 
\begin{equation} \label{eq:OPsfJnorm}
   H_{m,n}^{\b,\g} =  \frac{c_{2\a,\g}}{c_{2\a+2m,\g}} h_{n-m}^{(2\a + 2m,\g)}
\end{equation}
in terms of the norm of the Jacobi polynomial $h_m^{(a,b)}$ in \eqref{eq:JacobiNorm}.
\end{prop}

\begin{proof}
Since $Y_\ell^m$ are homogeneous polynomials, $Y_\ell^m (t\xi) = t^m Y_\ell^m(\xi)$. For $(x,t) \in\VV_0^{d+1}$, 
we let $x = t\xi$ with $\xi \in \sph$ and use the integral \eqref{eq:intVsf}. The proof follows exactly as in Proposition \ref{prop:OPconeJ}. In particular, the cardinality of the set $\{S_{m, \ell}^n:  0\le m \le n, \, 1 \le \ell \le \dim \CH_m^d\}$ 
is
$$
   \sum_{m=0}^n \dim  \CH_m^d = \sum_{m=0}^n \left[ \binom{m+d-1}{m} + \binom{m+d-2}{m-1}\right]
      = \binom{n+d-1}{n} + \binom{n+d-2}{n-1},
$$
which is equal to $\dim \CV_n(\VV_0^{d+1}, \varphi_{\b,\g})$, so that the set is an orthonormal basis
of $\CV_n(\VV_0^{d+1}, \varphi_{\b,\g})$.
\end{proof}

The polynomials $S_{m,\ell}^n$ defined in \eqref{eq:OPsfJ} are closely related to the polynomials 
$Q_{m,\kb}^n$ defined in \eqref{eq:coneJ} for $W_{\mu,\b,\g}$ on $\VV^{d+1}$. Indeed, if we choose the
orthogonal basis $P_{\kb}^m (\varpi_m)$ in \eqref{eq:coneJ} as the basis defined in \eqref{eq:basisBd}, 
then $Y_\ell^m$ is part of the basis and, consequently, $S_{m,\ell}^n$ is the restriction of the corresponding
$Q_{m,\kb}^n$ on the surface $\VV_0^{d+1}$. 

In particular, taking into account of Theorem \ref{thm:DEconeJ}, it is probably not surprising that there should
be a differential operator that has $\CV_n(\VV_0^{d+1}, \varphi_{\b,\g})$ as an eigenspace. What is surprising,
however, is that this holds only when $\b = -1$ as shown in the following theorem. 

\begin{thm} \label{thm:DEsfJ}
Let $d \ge 2$. Every $u\in \CV_n(\VV_0^{d+1}, \varphi_{-1,\g})$ satisfies the differential equation
\begin{equation}\label{eq:cone-eigenSf}
   D_{-1,\g} u =  -n (n+\g+d-1) u,
\end{equation}
where $D_{-1,\g} = D_{-1,\g}(x,t)$ is the second order linear differential operator
\begin{align*}
  D_{-1,\g} : = & \, t(1-t)\partial_t^2 + \big( d-1 - (d+\g)t \big) \partial_t+ t^{-1} \Delta_0^{(x)},
\end{align*}
where $\Delta_0^{(x)}$ is the Laplace-Beltrami operator in variable $x \in \sph$. 
\end{thm} 
 
\begin{proof}
We establish the result for $u = S_{m,\ell}^n$ in \eqref{eq:sfOPbasis}. For $x = t \xi$, we
write 
$$
S_{m,\ell}^n(x,t) = g(t) t^m Y_\ell^m (\xi), \quad  g(t) = P_{n-m}^{(\b+2m+d-1, \g)}(1- 2t).
$$ 
The differential equation for the Jacobi polynomial shows that $g$ satisfies \eqref{eq:diff1}
with $2\mu$ replaced by $\b$. Hence, a straightforward computation shows that the polynomial 
$f(t) = g(t)t^m$ satisfies the equation
\begin{align*} 
  & t(1-t) \frac{\d^2}{\d t^2} f(t) + \big( (\b+d) - (\b+\g+d+1)t \big) \frac{\d}{\d t} f(t)  \\ 
  & \qquad\qquad  = -n(n+\b+\g+d) f(t) +  m (m +\b + d-1) t^{-1} f(t).
\end{align*}
This identity also holds for $u(x,t) = f(t)Y_\ell^m(\xi)$. When $\b = -1$, the number in the last term, 
$m(m+\b+ d-1)$, is the eigenvalue of $Y_\ell^m$ for $- \Delta_0$. In particular, 
$$
 m (m+ d-2) t^{-1} f(t) Y_\ell^m = - t^{-1} f(t) \Delta_0^{(x)} Y_{\ell}^m = - t^{-1} \Delta_0^{(x)}S_{m,\ell}^n.
$$
Thus, $u$ satisfies \eqref{eq:cone-eigenSf}. The proof is complete. 
\end{proof}
 
\begin{rem}
The proof shows, evidently, that the polynomial $S_{m,\ell}^n$ satisfies a differential equation for 
$\b \ne -1$, but the equation depends on both $m$ and $n$ so that the corresponding differential
operator does not have $\CV_n(\VV_0^{d+1}, \varphi_{\b,\g})$ as an eigenspace.
\end{rem}

\subsection{Laguerre polynomials on the surface of the cone}
For $d \ge 2$, we consider $\VV^{d+1}$ with $b = \infty$ and choose $w$ as the Laguerre weight 
$$
     \varphi_{\b}(t) = t^\b e^{-t}, \qquad \b > -d.
$$ 
We then define the inner product on $\RR[x,t] / \la \|x\|^2 - t^2\ra$ in terms of $w= \varphi_{\b}$ by 
\begin{align*}
    \la f,g\ra_{\b} &=  b_{\b} \int_{\VV_0^{d+1}} f(x,t) g(x,t)  t^\b e^{-t} \d \sigma(x,t) \\
            & =  b_{\b} \int_0^\infty \int_{\sph}  f(t\xi,t) g(t\xi,t) \d \s(\xi) t^{\b+d-1} e^{-t} \d t. 
\end{align*}
Let $\{Y_\ell^m: 1\le \ell \le \dim \CH_m^d\}$ be an orthonormal basis of 
$\CH_m^d$. Then an orthogonal basis for $\CV_n(\VV_0^{d+1}, \varphi_\b)$ is given by 
\begin{equation}\label{eq:coneLsf}
    S_{m,\ell}^{n,L}x,t) = L_{n-m}^{2m + \b+d-1}(t) Y_\ell^m(x), \quad 0 \le m \le n, \quad 1\le \ell \le \dim \CH_m^d,
\end{equation}
in terms of the Laguerre polynomials $L_n^\a$. 

As in the case of Jacobi polynomials on the cone, these polynomials are eigenfunctions of a second
order differential operator only when $\b = -1$. 

\begin{thm} \label{thm:DEsfL}
Let $d \ge 2$. Every $u\in \CV_n(\VV_0^{d+1}, \varphi_{-1})$ satisfies the differential equation
\begin{equation}\label{eq:coneL-eigenSf}
   D_{-1} u := \left( t \partial_t^2 + \big( d-1 - t \big) \partial_t  + t^{-1} \Delta_0^{(x)}\right)u = - n  u,
\end{equation}
where $\Delta_0^{(x)}$ is the Laplace-Beltrami operator in variable $x \in \sph$. 
\end{thm} 
 
\begin{proof}
The proof is similar to that of \eqref{eq:cone-eigen}. The polynomial $L_{n-m}^{2m + \b+d-1}$ 
satisfies the differential equation \eqref{eq:diffL} with $2\mu$ replaced by $\a$, from which follows easily
that $L_{n-m}^{2m + \b+d-1}(t) t^m$ satisfies 
\begin{align*} 
  & t \frac{\d^2}{\d t^2} f(t) + (\b+d - t) \frac{\d}{\d t} f(t)   = -n f(t) +  m (m +\b + d-1) t^{-1} f(t).
\end{align*}
Rest of the proof follows exactly as in the proof of Theorem \ref{thm:DEsfJ}.
\end{proof}

\section{Reproducing kernels for the Jacobi polynomials on the surface of the cone} \label{sect:SF9}
\setcounter{equation}{0}

The reproducing kernel $\Pb_n(\varphi_{\b,\g})$ of the space $\CV_n(\VV_0^{d+1},\varphi_{\b,\g})$
is uniquely defined by the property that, for all $Y \in \CV_n(\VV_0^{d+1},\varphi_{\b,\g})$,
$$
   b_{\b,\g} \int_{\VV_0^{d+1}} \Pb_n \big(\varphi_{\b,\g}; (x,t),(y,s)\big) Y(y,s) \varphi_{\b,\g}(s) \d \s(y,s) = Y(x,t).
$$
The kernel can be written in terms of the orthogonal basis \eqref{eq:sfOPbasis} as 
\begin{align*}
  \Pb_n\big(\varphi_{\b,\g}; (x,t),(y,s)\big) = \sum_{m=0}^n \sum_{\ell=1}^{\dim \CH_m^d} \frac{S_{m,\ell}^n(x,t)S_{m,\ell}^n(y,s)}{H_{m,n}^{\b,\g}}.\end{align*}
Similar to the case of the solid cone, we first give an expression of this kernel in terms of the reproducing 
kernel $\Pb_n\big(\wh \varpi_{\a,\g}; (x,t),(y,s)\big)$, defined in \eqref{eq:PV-PT}, on the triangle $\VV^2$. 

\begin{thm} \label{thm:PnConeSf}
Let $d \ge 2$, $\b \ge - 1$ and $\g > -1$. For $(x,t), (y,s) \in \VV_0^{d+1}$, let $y = s y'$ with $y' \in \sph$. Then,
with $\a = \frac{\b+ d-1}{2}$, 
\begin{align}\label{eq:sfPbCone}
 \Pb_n \big(\varphi_{\b,\g}; (x,t), (y,s)\big) = c \!  \int_{-1}^1 \int_{-1}^1 &
        \Pb_n \left(\wh \varpi_{\a, \g}; \big( \tfrac{1-z_1}{2}\la x,y'\ra + \tfrac{1+z_1}{2} z_2t, t\big), (s, s)\right)\\
    &  \times   (1-z_1)^{\f{d-2}{2}} (1+z_1)^{\f{\b-1}{2}} (1-z_2^2)^\f{\b}{2} \d z, \notag
\end{align} 
where  $c= c_{\frac{d-2}{2}, \frac{\b -1}{2}} c_{\f{\b+1}{2}}$. In particular, for $\b = -1$, 
\begin{align}\label{eq:sfPbCone2}
   \Pb_n \big(\varphi_{-1,\g};  (x,t), (y,s)\big)  = \Pb_n \left(\wh \varpi_{\frac{d-2}{2}, \g}; (\la x,y'\ra ,t), (s, s)\right).
\end{align} 
\end{thm}

\begin{proof}
The spherical harmonics $Y_\ell^m$ in the basis $S_{m,\ell}^n$ 
are chosen to be an orthonormal basis of $\CH_m^d$, so that they satisfy the addition formula 
\eqref{eq:additionF}. Thus, in terms of the orthogonal basis $S_{m,\ell}^n$ of \eqref{eq:sfOPbasis}, 
the reproducing kernel satisfies 
\begin{align*}
\Pb_n(\varphi_{\b,\g}; (x,t), (y,s))  =  \sum_{m=0}^n & \frac{P_{n-m}^{(2\a+2m,\g)}(1-2t)P_{n-m}^{(2\a+2m,\g)}(1-2s)}
   {H_{m,n}^{\b,\g}} s^m t^m Z_m^{\f{d-2}{2}} (\la x',y'\ra),
\end{align*}
where $x = t x'$. Comparing \eqref{eq:coneJnorm} and \eqref{eq:OPsfJnorm}, and using \eqref{eq:trianglePV2} 
and \eqref{eq:norm=norm} with $\a = \frac{d-2}{2}$ or $\b = -1$, we have proved \eqref{eq:sfPbCone2}, where
we have used $\la x', y'\ra = \frac{\la x,y'\ra}{t}$.

For $\b > -1$, we follow the proof of Theorem \ref{thm:PnConeG} and use \eqref{eq:ZtoZ} to increase
the index of $Z_n^\l$ form $\l = \f{d-2}{2}$ to $\a = \l + \s$ with $\s = \frac{\b+1}{2}$. This gives
\begin{align*}
\Pb_n(\varphi_{\b,\g}; & (x,t), (y,s))  = c \int_{-1}^1 \int_{-1}^1
   \sum_{m=0}^n  \frac{P_{n-m}^{(2\a+2m,\g)}(1-2t)P_{n-m}^{(2\a+2m,\g)}(1-2s)} 
   {H_{m,n}^{\b,\g}} \\ 
    &\times  s^m t^m Z_m^{\a} \left( \tfrac{1-z_1}{2} \la x',y'\ra + \tfrac{1+z_1}{2} z_2 \right)
       (1-z_1)^{\f{d-2}{2}} (1+z_1)^{\f{\b-1}{2}} (1-z_2^2)^\f{\b}{2} \d z,
\end{align*}
where $c = c_{\frac{d-2}{2}, \frac{\b -1}{2}} c_{\f{\b+1}{2}}$. Hence, by \eqref{eq:norm=norm}, we can
again use \eqref{eq:trianglePV2} to derive \eqref{eq:sfPbCone}. 
\end{proof} 

Using the closed formula \eqref{eq:closedPV} of the reproducing kernels on the triangle, we can
now derive closed formulas of the reproducing kernels on the surface of the cone. 

\begin{thm}  \label{thm:sfPbCone2}
Let $d \ge 2$, $\b \ge -1$ and $\g \ge -\f12$. Let $\a = \frac{\b+ d-1}{2}$. Then, for $(x,t), (y,s) \in \VV_0^{d+1}$,
\begin{align} \label{eq:sfPbCone3}
 \Pb_n & \big(\varphi_{\b,\g}; (x,t), (y,s)\big) =  \wh c_{\b,\g} \int_{[-1,1]^4} Z_{2n}^{2\a+\g+1} \big(\xi (x,t,y,s; z,v)\big)\\
  &  \times   (1-z_1)^{\f{d-2}{2}} (1+z_1)^{\f{\b-1}{2}} (1-z_2^2)^\f{\b}{2}\d z (1-v_1^2)^{\a-1}(1-v_2^2)^{\g-\f12} \d v, \notag
\end{align} 
where $\wh c_{\b,\g}= c_{\frac{d-2}{2}, \frac{\b -1}{2}} c_{\f{\b+1}{2}} c_{\a-\f12} c_\g$ and  $\xi (x,t,y,s; z,v) \in [-1,1]$ is 
given by 
\begin{equation} \label{eq:sf-xi}
    \xi (x,t,y,s; z,v)  = \frac{v_1}2 \sqrt{2st + (1-z_1)\la x,y\ra + (1+z_1) z_2 s t} + v_2 \sqrt{1-t}\sqrt{1-s}.
\end{equation}
In particular, if $\b = -1$, then 
\begin{align} \label{eq:sfPbCone4}
 \Pb_n \big(\varphi_{-1,\g}; (x,t), (y,s)\big) = c  \int_{[-1,1]^2} Z_{2n}^{\g+d-1} &\left(v_1 \sqrt{\tfrac{st + \la x,y \ra}2}+ v_2 \sqrt{1-t}\sqrt{1-s} \right)\\
  &  \times  (1-v_1^2)^{\f{d-4}{2}} (1-v_2^2)^{\g-\f12} \d v. \notag
\end{align} 
These identities hold under the limit \eqref{eq:limit-int} when $\g = -\f12$.
\end{thm} 

\begin{proof}
Using \eqref{eq:closedPV} with $\zeta = \tfrac{1-z_1}{2}\la x,y'\ra + \tfrac{1+z_1}{2} z_2t$, 
\eqref{eq:sfPbCone3} is deduced from \eqref{eq:sfPbCone} by a straightforward substitution, just
as in the proof of Theorem \ref{thm:PnCone2}. 
\end{proof}

One particularly interesting case is the surface of cone in $\RR^3$ with $\g = -\f12$, the limiting 
case of \eqref{eq:sfPbCone4} when $d =3$ and $\g = -\f12$, which we stated as a corollary. 

\begin{cor}
For $d =2$, the reproducing kernel of $\CV_n(\VV^3, \varphi)$ for $\varphi(t) = t^{-1} (1-t)^{-\f12}$ satisfies
\begin{align}\label{eq:3Dcone}
& \Pb_n  \big(\varphi_{-1,-\f12}; (x,t), (y,s)\big)  \\
 &  =   Z_{2n}^{\f12}\left(\sqrt{\tfrac{st + \la x,y \ra}2}+\sqrt{1-t}\sqrt{1-s} \right)
 + Z_{2n}^{\f12}\left(- \sqrt{\tfrac{st - \la x,y \ra}2}+ \sqrt{1-t}\sqrt{1-s} \right) \notag \\
  & +  Z_{2n}^{\f12}\left( \sqrt{\tfrac{st + \la x,y \ra}2} - \sqrt{1-t}\sqrt{1-s} \right)
 + Z_{2n}^{\f12}\left(- \sqrt{\tfrac{st - \la x,y \ra}2}-\sqrt{1-t}\sqrt{1-s} \right). \notag
\end{align}\end{cor}

Recall that $Z_{n}^{\f12} = (2n+1) P_n$, where $P_n$ is the Legendre polynomial. This identity for
orthogonal polynomials on the surface of the cone in $\RR^3$ is similar to the addition formula for 
spherical harmonics. 

\section{Convolution and Fourier series on surface of the cone}
\setcounter{equation}{0}

We follow the development on the solid cone. All proofs can be carried out as in the case of
solid cone, most with minuscule modification, and will be omitted. 

Motivated by the closed formula \eqref{eq:sfPbCone}, we define a translation operator. 

\begin{defn}\label{defn:sfTmg} 
Let $d \ge 2$. For $\b \ge -1$ and $\g \ge -\f12$, let $\a = \frac{\b+d-1}{2}$. For $g\in L^1([-1,1],w_{2\a+\g+1})$, 
we define the operator $T_{\b,\g}$ on the surface of the cone $\VV_0^{d+1}$ by
\begin{align} \label{eq:sfTmg}
   T_{\b,\g} g\big((x,t),(y,s)\big) :=\, & \wh c_{\b,\g} \int_{[-1,1]^4}  g \left( \xi (x, t, y, s; u, v)\right)  (1-z_1)^{\f{d-2}{2}} (1+z_1)^{\f{\b-1}{2}}      \\
    &  \times  (1-z_2^2)^\f{\b}{2} \d z (1-v_1^2)^{\a-1}(1-v_2^2)^{\g-\f12} \d v, \notag
\end{align}
where $\xi(x,t,y,s;u,v)$ is defined by \eqref{eq:sf-xi}. When $\b = -1$ or $\g = -\f12$, the definition holds
under the limit \eqref{eq:limit-int}.
\end{defn}

By the closed formula of the reproducing kernel \eqref{eq:sfPbCone}, we immediately obtain
\begin{equation} \label{eq:sfPbZn}
   \Pb_n\big(\varphi_{\b,\g}; (x,t), (y,s)\big) = T_{\b,\g} Z_{2n}^{2\a + \g+ 1} \big((x,t),(y,s)\big).
\end{equation}
 
The following lemma is the analogue of Lemma \ref{lem:translateT} with essentially the same proof.

\begin{lem} \label{lem:translateTsf}
Let $d \ge 2$. For $\b \ge -1$ and $\g \ge -\f12$, let $\a = \frac{\b+d-1}{2}$. Let $g \in L^1([-1,1],w_{2\a+\g+1})$
be an even function on $[-1,1]$. Then 
\begin{enumerate}
\item for each $Y_n \in \CV_n (\VV_0^{d+1}, \varphi_{\b,\g})$, 
\begin{equation}\label{eq:sfFHcone}
  b_{\b,\g} \int_{\VV_0^{d+1}}  T_{\b,\g} g\big((x,t),(y,s)\big) Y_n (y,s) \varphi_{\b,\g}(s) \d \s(y,s)
       =  \Lambda_n (g) Y_n(x,t),
\end{equation}
where
$$
 \Lambda_n (g) = c_{2\a+\g+1} \int_{-1}^1 g(t) \frac{C_{2n}^{2\a+\g+1} (t)} {C_{2n}^{2\a+\g+1}(1)}
      (1-t^2)^{2 \a+\g + \f12} \d t.
$$
\item for $1\le p \le \infty$ and $(x,t) \in \VV_0^{d+1}$, 
\begin{equation}\label{eq:sfTbound}
  \left \| T_{\b,\g} g\big((x,t),(\cdot,\cdot)\big) \right\|_{L^p(\VV_0^{d+1}, \varphi_{\b,\g})}
   \le \|g \|_{L^p([-1,1],w_{2\a + \g+1})}.
\end{equation}
\end{enumerate}
\end{lem}

We then define an analogue of the convolution structure on the surface of the cone. 

\begin{defn}
Let $\b \ge -1$ and $\g \ge -\f12$ and let $\a = \frac{\b+d-1}{2}$. For $f \in L^1(\VV_0^{d+1}, \varphi_{\b,\g})$ 
and $g \in L^1([-1,1],w_{2\a+\g+1})$, define the convolution of $f$ and $g$ on the surface of the cone by 
$$
  (f \ast_{\b,\g} g)(x,t)  := b_{\b,\g} \int_{\VV_0^{d+1}} f(y,s) T_{\b,\g} g \big((x,t),(y,s)\big) \varphi_{\b,\g}(y,s) \d\s(y,s).
$$
\end{defn}

Again, the convolution on the cone satisfies Young's inequality:

\begin{thm} \label{thm:sfYoung}
Let $p,q,r \ge 1$ and $p^{-1} = r^{-1}+q^{-1}-1$. For $f \in L^q\left( \VV_0^{d+1}, \varphi_{\b,\g}\right)$ and
$g \in L^r(w_{2\a+\g+1};[-1,1])$ with $g$ an even function, 
\begin{equation} \label{eq:sfYoung}
              \|f \ast_{\b,\g} g\|_{L^p\left(\VV_0^{d+1}, \varphi_{\b,\g}\right)} 
                   \le \|f\|_{L^q\left( \VV_0^{d+1}, \varphi_{\b,\g}\right)} \|g\|_{L^r([-1,1],w_{2\a+\g+1})}.
\end{equation}
\end{thm}

An analogue of the Proposition \ref{prop:proj-conv} also holds. The projection operator
$\proj_n^{\b,\g}: L^2(\VV_0^{d+1}, \varphi_{\b,\g}) \mapsto \CV_n(\VV_0^{d+1}, \varphi_{\b,\g})$  
is defined by 
\begin{equation}\label{eq:sf-projCone}
   \proj_n^{\b,\g} f(x,t) = b_{\b,\g} \int_{\VV_0^{d+1}} f(y,s) \Pb_n\big(\varpi_{\b,\g}; (x,t), (y,s) \big ) 
      \varphi_{\b,\g}(y,s) \d \s(y, s).
\end{equation}
By \eqref{eq:sfPbCone2} and the definition of the convolution operator, we have 
$$
  \proj_n^{\b,\g} f = f  \ast_{\b,\g} Z_{2n}^{2\a+\g+1}. 
$$
Let $\Kb_n\left(\varphi_{\b,\g}; \cdot,\cdot \right)$ be the reproducing kernel of $\Pi_n^{d+1}$, which is the kernel
of the partial sum $S_n(\varphi_{\b,\g}; f)$ of the Fourier series in the Jacobi polynomials on $\VV_0^{d+1}$. 

\begin{prop} 
Let $\b \ge -\f12$, $\g \ge -\f12$ and $\a = \frac{\b+ d-1}{2}$. For $(x,t), (y,s) \in \VV_0^{d+1}$, 
\begin{align} \label{eq:sfKnCone}
   \Kb_n \left(\varphi_{\b,\g}; (x,t),(y,s)\right)  = T_{\b, \g}
    \left[k_n \big(w_{2\a + \g + \f12, -\f12}; 2\{ \cdot\}^2-1, 1\big)\right ]\big((x,t),(y,s)\big).  
\end{align}
Furthermore, for $(y,s) \in \VV_0^{d+1}$, 
\begin{align} \label{eq:sfKnCone2}
     \Kb_n \big(\varphi_{\b, \g}; (0,0),(y,s)\big)  = k_n \big(w_{2\a,\g}; 1-2s, 1\big).
\end{align}
\end{prop}

\begin{cor}
For $\b \ge 0$, $\g \ge -\f12$ and $\a = \f{\b+d-1}{2}$, 
\begin{equation} \label{eq:sfSnCone}
   S_n \left(\varpi_{\b,\g}; f\right) = f \ast_{\b,\g} k_n \big(w_{2\a + \g + \f12, -\f12}; 2\{ \cdot\}^2-1, 1\big). 
\end{equation}
\end{cor}

Finally, we can deduce the convergence of the Ces\`aro $(C,\delta)$ means as follows. 
 
\begin{thm}
For $\b \ge -1$ and $\g \ge -\f12$, define $\l_{\b,\g}: = \b+\g+d$. Then, the Ces\`aro $(C,\delta)$ 
means for $\varphi_{\b,\g}$ on $\VV_0^{d+1}$ satisfy 
\begin{enumerate} [\quad 1.]
\item if $\delta \ge \l_{\b,\g} + 1$, then $S_n^\delta(\varphi_{\b,\g}; f)$ is nonnegative if $f$ is nonnegative;
\item $S_n^\delta (\varphi_{\b,\g}; f)$ converge to $f$ in $L^1(\VV_0^{d+1}, \varphi_{\b,\g})$ norm or
$C(\VV_0^{d+1})$ 
norm if $\delta > \l_{\b,\g}$ and only if $\delta > \l_{\b,\g}$ when $\g = -\f12$. 
\end{enumerate}
\end{thm}

\section{Generalizations to reflection invariant weight functions}
\setcounter{equation}{0}

In our study so far, we have limited to the weight functions on the cone that arise from the classical weight
functions on the unit ball and the unit sphere, since they are the most interesting cases and they already 
capture the essence of our study. In this last section, we point out further extensions. 

\subsection{Spherical harmonics with reflection group invariance}
A profound extension of spherical harmonics is Dunkl's theory of $h$-harmonics associated with 
reflection groups \cite{D89, DX}. Let $G$ be a reflection group with a reduced root system $R$. Let 
$v \mapsto \k_v$ be a nonnegative multiplicity function defined on $R$ with the property that it is a constant 
on each conjugate class of $G$. Then the Dunkl operators \cite{D89} are defined by 
\begin{equation} \label{eq:Dunkl-Op}
    D_i f(x) = \partial_i  f(x) + \sum_{v \in R_+} \k_v \frac{f(x) - f (x \s_v)}{\la x,\s_v \ra} v_i,  \qquad i =1,2 , \ldots,d,
\end{equation}
where $x \s_v := x - 2\la x, v \ra v /\|v\|^2$, $R_+$ is a set of positive roots. These first order differential-difference operators commute in the sense that $D_i D_j = D_j D_i$ for $1 \le i, j \le d$. The operator 
$$
\Delta_h = D_1^2+ \cdots + D_d^2
$$
plays the role of the Laplace operator. An $h$-harmonic polynomial of degree $n$ is a polynomial $Y \in \CP_n^d$ 
that satisfies $\Delta_h f = 0$. Such polynomials are orthogonal with respect to the inner product \cite{D91}
\begin{equation} \label{eq:Dunkl-hk}
  \la f,g\ra_\k = a_\k \int_\sph f(x) g(x) h_\k^2(x) \d\s(x), \qquad h_\k(x) = \prod_{v\in R_+} |\la x,v\ra|^{\k_v},
\end{equation}
on the sphere, where $a_\k$ is a normalization constant so that $\la 1,1\ra_\k =1$. Let $\CH_n^d(h_\k^2)$
be the space of $h$-harmonics of degree $n$, which has the same dimension as that of $\CH_n^d$. 
Let $\Delta_{h,0}$ be the restriction of $\Delta_h$ on the sphere, which is the analog of the Laplace--Beltrami
operator. Then $\CH_n^d(h_\k^2)$ is the eigenspace of the operator $\Delta_{h,0}$:  
\begin{equation} \label{eq:Delta_h0}
   \Delta_{h,0} Y^h = - n (n+2 \l_\k) Y^h, \qquad \l_\k := |\k| + \tfrac{d-2}{2},
\end{equation}
where $|\k| = \sum_{v\in R_+} \k_v$. A linear operator, denoted by $V_\kappa$, that satisfies the relations 
\begin{equation*} 
    D_i V_\k  = V_\k \partial_i, \qquad 1 \le i \le d,
\end{equation*}
is called an intertwining operator, which is uniquely determined if it also satisfies $V_\k 1 = 1$ and 
$V_\k \CP_n^d \subset \CP_n^d$. Let $\Pb(h_\k^2; \cdot,\cdot)$ be the reproducing kernel of 
$\CH_n^d(h_\k^2)$. The kernel has a closed form in terms of the intertwining operator $V_\k$,
\begin{equation} \label{eq:Pb-hk}
\Pb(h_\k^2; x, y) = V_\k \left[Z_n^{\l_\k} (\la \cdot, y\ra)\right](x), \qquad x,y \in \sph,
\end{equation}
where $Z_n^\l$ is defined in \eqref{eq:Zn}. 
The explicit formula of $V_\k$ is known only in the case of $G = \ZZ_2^d$ and the weight function
\begin{equation}\label{eq:hkZ2d} 
     h_\k(x) = \prod_{i=1}^d |x_i|^{\k_i}, \qquad \k_i \ge 0 
\end{equation}
invariant under $\ZZ_2^d$. In this case, the intertwining operator is a multiple beta integral, which 
gives in particular the closed formula for the reproducing kernel \cite{X97}
\begin{equation} \label{eq:PbZ2d}
     \Pb(h_\k^2; x, y) = c_\k^h \int_{[-1,1]^d} C_n^{\l_\k} (x_1 y_1 t_1 + \cdots + x_d y_d t_d) 
        \prod_{i=1}^d (1+t_i)(1-t_i^2)^{\k_i-1} \d t,
\end{equation}
where $c_k^h = c_{\k_1-\f12} \cdots c_{\k_d-\f12}$ with $c_\l$ defined by \eqref{eq:c_l}.

\subsection{Orthogonal structure on the surface of the cone}
On the surface of the cone $\VV_0^{d+1}$, we can extend the inner product \eqref{eq:innerSF} to 
$$
  \la f,g\ra_{\k,w} (t) = b_{\k} \int_{\VV_0^{d+1}}  f(x,t) g(x,t) h_\k^2(x) w(t) d\s(x,t)
$$
with $h_\k$ in \eqref{eq:Dunkl-hk} invariant under a reflection group. Much of what we have done in
Section 7 can be extend in this more general setting with $\a = \frac12(\b+ d-1)$ replaced by 
$\a = \frac12(\b+|\k| + d-1)$.  

For example, with the Jacobi weight function $w= \varphi_{\b,\g}$, we consider 
$$
   \varPhi_{\k,\b,\g}(t) = h_\k^2(x) t^\b(1-t)^\g, \qquad (x,t) \in \VV_0^{d+1}. 
$$
Similar to \eqref{eq:sfOPbasis}, we can give an orthogonal basis of $\CV_n(\VV_0^{d+1}, \varPhi_{\k,\b,\g})$ 
by 
\begin{equation} \label{OPbasis-refl}
  S_{m, \ell}^n (x,t) = P_{n-m}^{(2m + 2\a,\g)} (t) Y_\ell^m (x), \quad 0 \le m \le n, \,\, 1 \le \ell \le \dim \CH_m^d(h_\k^2),
\end{equation}
where $\{Y_\ell^m: 1 \le \ell \le \dim \CH_m^d\}$ denote an orthonormal basis of $\CH_m^d(h_\k^2)$.

\begin{thm} \label{thm:sfDEJ-refl}
Let $d \ge 2$. Every $u\in \CV_n(\VV_0^{d+1}, \varPhi_{\k, -1,\g})$ satisfies the differential-difference
equation
\begin{equation}\label{eq:eigenSf-refl}
   D_{\k,-1, \g} u =  -n (n+ |\k| + \g+ d-1) u,  
\end{equation}
where $D_{\k,-1, \g} = D_{\k,-1,\g}(x,t)$ is the second order linear differential-difference operator
\begin{align*}
D_{\k,-1, \g} : = & \, t(1-t)\partial_t^2 + \big(|\k| + d-1 - (|\k| + d+\g)t \big) + t^{-1} \Delta_{h,0}^{(x)},
\end{align*}
where $\Delta_{h,0}^{(x)}$ is the operator in variable $x \in \sph$. 
\end{thm} 

\begin{proof}
The proof follows that of Theorem \ref{thm:DEsfJ} almost verbatim if we replace $\b$ by $\b+|\k|$ and use
\eqref{eq:Delta_h0} instead of \eqref{eq:sph-harmonics} in the last step. 
\end{proof}

An analog of Theorem \ref{thm:DEsfL} also holds. Furthermore, we can also state a closed formula for
the reproducing kernel of $\CV_n(\VV_0^{d+1}, \varPhi_{\k,\b,\g})$, which we state only for 
$h_\k$ in \eqref{eq:hkZ2d} invariant under $\ZZ_2^d$, because of the explicit formula \eqref{eq:PbZ2d}. 

\begin{thm}  
Let $d \ge 2$, $\k_i \ge 0$, $\b \ge -1$ and $\g \ge -\f12$. Let $\a = \frac12 (|\k|+\b+ d-1)$. 
Then, for $(x,t), (y,s) \in \VV_0^{d+1}$,
\begin{align} \label{eq:sfPbCone-refl}
 \Pb_n & \big(\varPhi_{\k,\b,\g}; (x,t), (y,s)\big) = 
    \wh c_{\b,\g} \int_{[-1,1]^{d+4}} Z_{2n}^{2\a+\g+1} \big(\xi (x,t,y,s; z,v,u)\big)\\
  &  \times   (1-z_1)^{\f{d-2}{2}} (1+z_1)^{\f{\b-1}{2}} (1-z_2^2)^\f{\b}{2} \d z (1-v_1^2)^{\a-1}(1-v_2^2)^{\g-\f12} \d v
     \notag \\
  &  \times     \prod_{i=1}^d (1+u_i)(1-u_i^2)^{\k_i-1} \d u, \notag
\end{align} 
where $\wh c_{\b,\g}= c_\k^h c_{\frac{d-2}{2}, \frac{\b -1}{2}} c_{\f{\b+1}{2}} c_{\a-\f12} c_\g$ and  
$\xi (x,t,y,s; z,v) \in [-1,1]$ is given by 
\begin{align*} 
    \xi (x,t,y,s; z,v)  = & \frac{v_1}2 \sqrt{2st + (1-z_1) (x_1y_1 u_1 +\cdots x_dy_d u_d) + (1+z_1) z_2 s t}\\ 
           & + v_2 \sqrt{1-t}\sqrt{1-s}, 
\end{align*}
and the identity holds under the limit \eqref{eq:limit-int} whenever $\b=-1$, $\g = -\f12$ or $\k_i = 0$.
\end{thm}

Although the closed formula \eqref{eq:sfPbCone-refl} is complicated, its one-dimensional structure allows
us to carry out the narrative that we developed so far for the more general weight functions on the surface
$\VV_0^{d+1}$. We could define the translation operator and the convolution in this general setting, 
and establish similar properties as those in Section \ref{sect:SF9}. Instead of stating the results, which
carries little additional difficulty, we shall state only a result for the convergence of the Ces\`aro means. 

\begin{thm}
For $\k_i \ge 0$, $\b \ge -1$ and $\g \ge -\f12$, define $\l_{\k,\b,\g}: =|\k|+\b+\g+d$. Then, the Ces\`aro $(C,\delta)$ 
means for $\varPhi_{\k,\b,\g}$ on $\VV_0^{d+1}$ satisfy 
\begin{enumerate} [\quad 1.]
\item if $\delta \ge \l_{\b,\g} + 1$, then $S_n^\delta(\varPhi_{\k,\b,\g}; f)$ is nonnegative if $f$ is nonnegative;
\item $S_n^\delta (\varPhi_{\k,\b,\g}; f)$ converge to $f$ in $L^1(\VV^{d+1}, \varphi_{\b,\g})$ 
norm or $C(\VV^{d+1})$ norm if $\delta > \l_{\b,\g}$.
\end{enumerate}
\end{thm}

\subsection{Orthogonal structure on the cone}
On the unit ball $\BB^d$ we can also consider the reflection invariant weight function 
$$
       \varpi_{\k,\mu}(x) = h_\k^2(x) (1-\|x\|^2)^{\mu-\f12} 
$$
for an $h_\k$ invariant under a reflection group $G$, and study orthogonal polynomials with respect 
to $\varpi_{\k,\mu}$ on the ball \cite[Section 8.1]{DX}. In this setting, the space $\CV_n^d(\varpi_{\k,\mu})$
is the eigenspace of a second order differential-difference operator 
\cite[Theorem 8.1.3]{DX}: for all $ u \in \CV_n^d(\varpi_{\k,\mu})$, 
\begin{equation}\label{eq:diff-ball-refl}
   \left(\Delta_h - \la x,\nabla\ra^2 - 2 \l_{\k,\mu} \la x,\nabla\ra \right)u = - n (n+2\l_{\k,\mu}) u, \quad 
     \l_{\k,\mu} = |\k| + \mu + \tfrac{d-1}{2}. 
\end{equation}
Furthermore, the reproducing kernel $\Pb(\varpi_{\k,\mu}; \cdot,\cdot)$ of $\CV_n^d(\varpi_{\k,\mu})$
satisfies a closed form in terms of the intertwining operator $V_\k$, which we state only for the 
case of $G=\ZZ_2^d$. For $h_k$ in \eqref{eq:hkZ2d}, we have  \cite[Theorem 8.1.6]{DX} 
\begin{align} \label{eq:Pb-ball-refl}
\Pb(\varpi_{\k,\mu}; x, y) = c_\k^h c_\mu \int_{[-1,1]^{d+1}} Z_n^{\l_{\k,\mu}} ( x_1y_1 t_1 +\cdots x_d y_d t_d
        +  x_{d+1} y_{d+1} s ) \\
       \times \prod_{i=1}^d (1+t_i)(1-t_i^2)^{\k_i-1} (1-s^2)^{\mu-\f12}\d t \d s, \notag
\end{align}
where $x_{d+1} = \sqrt{1-\|x\|^2}$ and $y_{d+1} = \sqrt{1-\|y\|^2}$. 

On the cone $\VV_0^{d+1}$, we can replace $\varpi_\mu$ by $\varpi_{\k,\mu}$ to study orthogonal 
structure for more general weight functions. For example, instead of the weight function \eqref{eq:W-cone},
we can consider 
$$
  W_{\k, \mu,\b, \g}(x,t): = h_\k^2(x) (t^2-\|x\|^2)^{\mu-\f12} t^\b (1-t)^\g, \quad   \mu > -\tfrac12, \quad \g > -1, 
$$
with $h_\k$ as in \eqref{eq:Dunkl-hk} invariant under a reflection group. Most of our results in Section 3 and 4 
can be extended in this more general setting with $\a = \mu+ \frac12(\b+ d-1)$ replaced by 
$\a = |\k|+ \mu+ \frac12(\b+ d-1)$.  In particular, let $\{P_\kb^{n-m}(\varpi_{\k,\mu};x): |\kb|=n\}$ be an orthonormal 
basis of $\CV_{n-m}^d(\varpi_{\kb,\mu})$ on the unit ball $\BB^d$. Then an orthogonal basis of 
$\CV_n(\VV^{d+1},  W_{\k, \mu,\b, \g})$, analogous to that of \eqref{eq:coneJ}, is given by 
\begin{equation} \label{eq:coneJ-refl}
  Q_{m, \kb}^n (x,t) = P_{n-m}^{(2\a + 2m,\g)} (1-2t) t^m P_\kb^n(\varpi_{\k,\mu};x), \quad 0 \le m \le n, \,\, |\kb| =n-m.
\end{equation}

In the case of $\b = 0$, we have the following extension of Theorem \ref{thm:DEconeJ}:

\begin{thm} \label{thm:DEconeJ-refl} 
For $\mu > -\tfrac12$, $\g > -1$ and $\k_i \ge 0$, define $\a = |\k|+ \mu +\frac{d-1}{2}$. Then 
every $u \in \CV_n(\VV^{d+1},W_{\k, \mu,0, \g})$ satisfies the differential-difference equation
\begin{equation}\label{eq:cone-eigen-refl}
   \CD_{\k,\mu,\g} u =  -n (n+ 2\a+\g+1) u,
\end{equation}
where $\CD_{\k,\mu,\g} = \CD_{\k, \mu,\g}(x,t)$ is the second order linear differential-difference operator
\begin{align*}
  \CD_{\k, \mu,\g} : = &     t(1-t)\partial_{tt} +  2 (1-t) \la x,\nabla_x \ra \partial_t + t \Delta_h^{(x)} - 
       \la x,\nabla_x \ra^2    \\ 
  &     + (2\a+1)\partial_t - (2\a+\g+2 ) \big( \la x,\nabla_x\ra + t \partial_t\big) +\la x,\nabla_x\ra, 
\end{align*}
where $\nabla_x$ and $\Delta_h^{(x)}$ denote the operators acting on the $x$ variable. 
\end{thm} 

\begin{proof}
The proof follows that of Theorem \ref{thm:DEconeJ} almost verbatim if we replace $\mu$ by $\mu+|\k|$
and use \eqref{eq:diff-ball-refl} instead of \eqref{eq:diffBall}. We choose not, however, to carry out the
last step of expanding $\la x,\nabla_x\ra^2$, which would lead to a final expression of $\CD_{\k,\mu,\g}$ 
that contains both $\Delta_x$ and $\Delta_h^{(x)}$. 
\end{proof}

We can also derive a closed formula for the reproducing kernel $\Pb_n(W_{\k, \mu, \b, \g}; \cdot,\cdot)$
of the space $\CV_n(\VV^{d+1},W_{\k, \mu, \b, \g})$, which we again state only for $h_\k$ in \eqref{eq:hkZ2d} 
invariant under $\ZZ_2^d$, using the explicit formula \eqref{eq:PbZ2d}. 

\begin{thm} \label{thm:PbCone2G-refl}
Let $d \ge 2$. For $\mu \ge 0$, $\b \ge -\f12$, $\g \ge -\f12$ and $\k_i \ge 0$, let 
$\a = |\k| + \mu + \frac{\b+d-1}{2}$. Then 
\begin{align*}
  \Pb_n \big(W_{\k, \mu,\b, \g}; &\, (x,t), (y,s)\big)  
 \wh c_\k^h c_{\mu,\b}  c_{\a -\f{1}{2}} c_\g \int_{[-1,1]^{d+5}}  Z_{2n}^{2 \a+\g+1} (\wh \xi (x, t, y, s; z, u, v, p)) \\
 &\times \prod_{i=1}^d (1+p_i)(1-p_i^2)^{\k_i-1} (1-z_1)^{\mu+\f{d-1}{2}} (1+z_1)^{\f{\b}2 -1} (1-z_2^2)^{\f{\b-1}{2}}  \notag \\  
     & \times (1-u^2)^{\mu-1} du (1-v_1^2)^{\a - 1}(1-v_2^2)^{\g-\f12}  \d p \d z \,\d u\, \d v \notag
\end{align*}
where $\wh \xi (x,t, y,s; u, v,p,q) \in [-1,1]$ is defined by 
\begin{align*}
 \wh \xi (x,t, y,s; u, v)  =  v_2 \sqrt{1-t}\sqrt{1-s}  +\frac12 v_1 \sqrt{\rho(x,t,y,s,p, z)},
\end{align*}
in which 
\begin{align*}
  \rho(x,t,y,s,p, z) = &\, 2 s t + (1+z_1)z_2 st \\
        & + (1-z_1) \left( p_1 x_1 y_1+\cdots+ p_d x_d y_d + \sqrt{t^2-\|x\|^2} \sqrt{s^2-\|y\|^2} \, u \right). 
 \end{align*}
In the case $\mu =  0$ or $\b = -\f12$ or $\g = -\f12$ or $\k_i=0$, the identity holds under the 
limit \eqref{eq:limit-int}. 
\end{thm}

As in the previous subsection, we have no difficulty to carry out our narrative for the solid cone $\VV^{d+1}$
to more general weight functions $W_{\k,\mu,\b,\g}$, but choose not to state them. The corresponding result
for the Ces\`aro means of the Fourier orthogonal series is given below: 

\begin{thm}
For $\k \ge0, \mu \ge 0$, $\b \ge 0$ and $\g\ge -\f12$, define $\l_{\mu,\b,\g}: = 2\mu+ |\k| + \b+\g+d$. Then, 
the Ces\`aro $(C,\delta)$ means for $W_{\k, \mu,\b,\g}$ on $\VV^{d+1}$ satisfy 
\begin{enumerate} [\quad 1.]
\item if $\delta \ge \l_{\k,\mu,\b,\g} + 1$, then $S_n^\delta(W_{\k,\mu,\b,\g}; f)$ is nonnegative if $f$ is nonnegative;
\item $S_n^\delta (W_{\k,\mu,\b,\g}; f)$ converge to $f$ in $L^1(\VV^{d+1}, W_{\k,\mu,\b,\g})$ norm 
or $C(\VV^{d+1})$ norm if $\delta > \l_{\k, \mu,\b,\g}$.
\end{enumerate}
\end{thm}

\bigskip\noindent
{\bf Acknowledgement}. I thank an anonymous referee for her/his numerous corrections and helpful suggestions.


\begin{thebibliography}{99}

\bibitem{AF}
        P. Appell and M. J. Kamp\'e de F\'eriet, 
	\textit{Fonctions hyperg\'eom\'etriques et hypersph\'eriques, 
        polynomes d'Hermite}, Gauthier-Villars, Paris, 1926.

\bibitem{DaiX}
        F. Dai and Y. Xu,
        \textit{Approximation theory and harmonic analysis on spheres and balls},
        Springer Monographs in Mathematics, Springer, 2013. 

\bibitem{D89}  
	C. F. Dunkl, 
	Differential-Difference operators associated to reflection groups.
	\textit{Trans. Amer. Math. Soc.} \textbf{311} (1989), 167--183.

\bibitem{D91}  
	C. F. Dunkl, 
	Integral kernels with reflection group invariance,
	\textit{Can. J. Math.} \textbf{43} (1991), 1213--1227.
	
\bibitem{DX} 
        C. F. Dunkl and Y. Xu,
        \textit{Orthogonal Polynomials of Several Variables}, 2nd ed., 
        Encyclopedia of Mathematics and its Applications \textbf{155},
         Cambridge University Press, Cambridge, 2014.
   
\bibitem{Gas}
        G. Gasper, 
        Positive sums of the classical orthogonal polynomials, 
        \textit{SIAM J. Math. Anal.} \textbf{8} (1977), 423--447.
       
\bibitem{KS}
        H. L. Krall and I. M. Sheffer,
	Orthogonal polynomials in two variables, 
	\textit{Ann. Mat. Pura Appl.} (4) \textbf{76} (1967), 325--376.

\bibitem{KPX}
        G. Kyriazis, P. Petrushev and Y. Xu, 
        Decomposition of weighted Triebel-Lizorkin and Besov spaces on the ball, 
        \textit{Proc. London Math. Soc.} \textbf{97} (2008), 477--513.
        
\bibitem{OX3}
        S. Olver and Y. Xu, 
        Orthogonal polynomials in and on a quadratic surface of revolution. 
        \textit{submitted} arXiv:1906.12305
        
\bibitem{PX} 
        P. Petrushev and Y. Xu,
        Localized polynomial frames on the ball,
        \textit{Constr. Approx.} \textbf{27} (2008) 121--148.

\bibitem{Stein} 
         E. M. Stein and G. Weiss, 
         {\it Introduction to Fourier Analysis on Euclidean Spaces},
         Princeton University Press, Princeton, 1971.

\bibitem{Sz}
       G. Szeg\H{o},
       \textit{Orthogonal polynomials}, 4th edition,
       Amer. Math. Soc., Providence, RI. 1975. 
  
\bibitem{Thanga}
        	S. Thangavelu, 
	\textit{Lectures on Hermite and Laguerre expansions},
	Princeton University Press, Princeton, NJ, 1993.
	
\bibitem{X97}
        Y. Xu, 
        Orthogonal polynomials for a family of product weight functions on the spheres, 
        \textit{Can. J. Math.}, \textbf{49} (1997), 175--192.

\bibitem{X99}
        Y. Xu,
        Summability of Fourier orthogonal series for Jacobi weight on a
	ball in $\RR^d$,
        \textit{Trans. Amer. Math. Soc.} \textbf{351} (1999), 2439--2458.

\bibitem{X15}
        Y. Xu, 
        An integral identity with applications in orthogonal polynomials, 
        \textit{Proc. Amer. Math. Soc.}, \textbf{143} (2015), 5253--5263.
    
\end{thebibliography}
\end{document}